\documentclass[10pt,english]{article}
\usepackage[OT1]{fontenc} 
\usepackage[latin1]{inputenc}
\usepackage{graphicx}
\usepackage{color}
\usepackage{dsfont}
\usepackage{amsmath,amsfonts,amsthm,amssymb,amscd,amssymb}
\usepackage{hyperref}
\usepackage{geometry}
\usepackage{amsmath,amssymb}
\usepackage{mathrsfs}
\usepackage{mathtools}

\usepackage{graphicx}

\usepackage{stmaryrd}

\usepackage{multibib}

\usepackage{multicol}
\usepackage{bbm}

\usepackage{calc}

\usepackage{float}
\usepackage{verbatim}

\usepackage{enumerate}

\usepackage{subfig}
\usepackage{color}
\usepackage{framed}
\usepackage{ulem}

\usepackage{relsize}

\newcommand{\R}{\mathbb{R}}

\newcommand{\p}{\mathbb{P}}
\newcommand{\E}{\mathbb{E}}

\newcommand{\X}{\mathbb{X}}
\newcommand{\Y}{\mathbb{Y}}

\newcommand{\dd}{{d}}

\newcommand{\norm}[1]{\left\Vert #1\right\Vert}

\newcommand{\argmax}{\operatornamewithlimits{argmax}}

\newtheorem{thm}{Theorem}
\newtheorem{lem}[thm]{Lemma}
\newtheorem{Def}[thm]{Definition}
\newtheorem{prop}[thm]{Proposition}

\newtheorem{cor}[thm]{Corollary}
\newtheorem{rem}[thm]{Remark}

\newcommand{\proofof}{\text{of }}

\begin{document}
\title{Stability and Minimax Optimality of Tangential Delaunay Complexes for Manifold Reconstruction
}

\date{\vspace{-5ex}}
\author{
Eddie Aamari \and Cl\'ement Levrard 
}


%

\maketitle

\begin{abstract}
\noindent We consider the problem of optimality in manifold reconstruction. A random sample $\mathbb{X}_n = \left\{X_1,\ldots,X_n\right\}\subset \R^D$  composed of points close to a $d$-dimensional submanifold $M$, with or without outliers drawn in the ambient space, is observed. Based on the Tangential Delaunay Complex \cite{Boissonnat14}, we construct an estimator $\hat{M}$ that is ambient isotopic and Hausdorff-close to $M$ with high probability. The estimator $\hat{M}$ is built from existing algorithms. In a model with additive noise of small amplitude, we show that this estimator is asymptotically minimax optimal for the Hausdorff distance over a class of submanifolds satisfying a reach constraint. Therefore, even with no \textit{a priori} information on the tangent spaces of $M$, our estimator based on Tangential Delaunay Complexes is optimal. This shows that the optimal rate of convergence can be achieved through existing algorithms. A similar result is also derived in a model with outliers.
A geometric interpolation result is derived, showing that the Tangential Delaunay Complex is stable with respect to noise and perturbations of the tangent spaces. In the process, a decluttering procedure and a tangent space estimator both based on local principal component analysis (PCA) are studied.
\end{abstract}


\section{Introduction}
Throughout many fields of applied science, data in $\R^D$ can naturally be modeled as lying on a $d$-dimensional submanifold $M$.
As $M$ may carry a lot of information about the studied phenomenon, it is then natural to consider the problem of either approximating $M$ geometrically, recovering it topologically, or both from a point sample $\X_n = \left\{X_1,\ldots,X_n\right\}$.
It is of particular interest in high codimension ($d \ll D$) where it can be used as a preliminary processing of the data for reducing its dimension, and then avoiding the curse of dimensionality.
This problem is usually referred to as \textit{manifold reconstruction} in the computational geometry community, and rather called \textit{set/support estimation} or \textit{manifold learning} in the statistics literature.

The computational geometry community has now been active on manifold reconstruction for many years, mainly in deterministic frameworks.
In dimension $3$, \cite{DeyBook} provides a survey of the state of the art. In higher dimension, the employed methods rely on variants of the ambient Delaunay triangulation \cite{Dey05,Boissonnat14}. The geometric and topological guarantees are derived under the assumption that the point cloud --- fixed and nonrandom --- densely samples $M$ at scale $\varepsilon$, with $\varepsilon$ small enough or going to $0$.

In the statistics literature, most of the attention has been paid to approximation guarantees, rather than topological ones. The approximation bounds are given in terms of the sample size $n$, that is assumed to be large enough or going to infinity. 
To derive these bounds, a broad variety of assumptions on $M$ have been considered.
For instance, if $M$ is a bounded convex set and $\X_n$ does not contain outliers, a natural idea is to consider the convex hull $\hat{M} = Conv(\X_n)$ to be the estimator. $Conv(\X_n)$ provides optimal rates of approximation for several loss functions \cite{Tsybakov95,Dumbgen96}. These rates depend crudely on the regularity of the boundary of the convex set $M$. In addition, $Conv(\X_n)$ clearly is ambient isotopic to $M$ so that it has both good geometric and topological properties.
Generalisations of the notion of convexity based on rolling ball-type assumptions such as $r$-convexity and reach bounds \cite{Cuevas2004,Genovese12} yield rich classes of sets with good geometric properties. In particular, the reach, as introduced by Federer \cite{federer1959}, appears to be a key regularity and scale parameter {\color{black}{\cite{Chazal06,Genovese12,Maggioni16}}}.

This paper mainly follows up the two articles \cite{Boissonnat14,Genovese12}, both dealing with the case of a $d$-dimensional submanifold $M \subset \R^D$ with a reach regularity condition and where the dimension $d$ is known.

On one hand, \cite{Boissonnat14} focuses on a deterministic analysis and proposes a provably faithful reconstruction. The authors introduce a weighted Delaunay triangulation restricted to tangent spaces, the so-called Tangential Delaunay Complex. This paper gives a reconstruction up to ambient isotopy with approximation bounds for the Hausdorff distance along with computational complexity bounds. This work provides a simplicial complex based on the input point cloud and tangent spaces. However, it lacks stability up to now, in the sense that the assumptions used in the proofs of \cite{Boissonnat14} do not resist ambient perturbations. Indeed, it heavily relies on the knowledge of the tangent spaces at each point and on the absence of noise.

On the other hand, \cite{Genovese12} takes a statistical approach in a model possibly corrupted by additive noise, or containing outlier points. The authors derive an estimator that is proved to be minimax optimal for the Hausdorff distance $\dd_{H}$. Roughly speaking, minimax optimality of the proposed estimator means that it performs best in the worst possible case up to numerical constants, when the sample size $n$ is large enough.
Although theoretically optimal, the proposed estimator appears to be intractable in practice. At last, \cite{Maggioni16} proposes a manifold estimator based on local linear patches that is tractable but fails to achieve the optimal rates. 

\subsection*{Contribution}
Our main contributions (Theorems \ref{main_result_noob_nonoise}, \ref{main_result_noob_noise} and \ref{theoretical_rate}) make a two-way link between the approaches of \cite{Boissonnat14} and~\cite{Genovese12}.

From a geometric perspective, Theorem \ref{main_result_noob_nonoise} shows that the Tangential Delaunay Complex of \cite{Boissonnat14} can be combined with  local PCA to provide a manifold estimator that is optimal in the sense of \cite{Genovese12}.
This remains possible even if data is corrupted with additive noise of small amplitude.
Also, Theorems \ref{main_result_noob_noise} and \ref{theoretical_rate} show that, if outlier points are present (clutter noise), the Tangential Delaunay Complex of \cite{Boissonnat14} still yields the optimal rates of \cite{Genovese12}, at the price of an additional decluttering procedure.

 From a statistical point of view, our results show that the optimal rates described in \cite{Genovese12} can be achieved by a tractable estimator $\hat{M}$  that (1) is a simplicial complex of which vertices are the data points, and (2) such that $\hat{M}$ is ambient isotopic to $M$ with high probability.

In the process, a stability result for the Tangential Delaunay Complex (Theorem \ref{tdc_stability}) is proved.
Let us point out that this stability is derived using an interpolation result (Theorem \ref{tangential_perturbation_thm}) which is interesting in its own right. 
Theorem \ref{tangential_perturbation_thm} states that 
if a point cloud $\mathcal{P}$ lies close to a submanifold $M$, and that estimated tangent spaces at each sample point are given, then there is a submanifold $M'$ (ambient isotopic, and close to $M$ for the Hausdorff distance) that interpolates $\mathcal{P}$, with $T_p M'$ agreeing with the estimated tangent spaces at each point $p \in \mathcal{P}$.
Moreover, the construction can be done so that the reach of $M'$ is bounded in terms of the reach of $M$, provided that $\mathcal{P}$ is sparse, points of $\mathcal{P}$ lie close to $M$, and error on the estimated tangent spaces is small.
Hence, Theorem \ref{tangential_perturbation_thm} essentially allows to consider a noisy sample with estimated tangent spaces as an exact sample with exact tangent spaces on a proxy submanifold.
This approach can provide stability for any algorithm that takes point cloud and tangent spaces as input, such as the so-called \textit{cocone} complex \cite{Dey05}.

\section*{Outline}
This paper deals with the case where a sample $\X_n = \left\{X_1,\ldots,X_n\right\}\subset \R^D$ of size $n$ is randomly drawn on/around $M$. 
First, the statistical framework is described (Section \ref{statistical_model_section}) together with minimax optimality (Section \ref{subsec:minimax_risk}). Then, the main results are stated (Section \ref{subsec:main_results}).

Two models are studied, one where $\X_n$ is corrupted with additive noise, and one where $\X_n$ contains outliers.
We build a simplicial complex $\hat{M}_{\mathtt{TDC}}(\X_n)$ ambient isotopic to $M$ and we derive the rate of approximation for the Hausdorff distance $\dd_{H}(M,\hat{M}_{\mathtt{TDC}})$, with bounds holding uniformly over a class of submanifolds satisfying a reach regularity condition. The derived rate of convergence is minimax optimal if the amplitude $\sigma$ of the additive noise is small. 
With outliers, similar estimators $\hat{M}_{\mathtt{TDC \delta}}$ and $\hat{M}_{\mathtt{TDC +}}$ are built.
$\hat{M}_{\mathtt{TDC}}$, $\hat{M}_{\mathtt{TDC \delta}}$ and $\hat{M}_{\mathtt{TDC +}}$ are based on the Tangential Delaunay Complex (Section \ref{tangential_delaunay_complex_section}), 
that is first proved to be stable (Section \ref{stability_section}) via an interpolation result. 
A method to estimate tangent spaces and to remove outliers based on local Principal Component Analysis (PCA) is proposed (Section \ref{tangent_space_estimation_and_denoising_procedure}). We conclude with general remarks and possible extensions (Section \ref{conclusion}).
{\color{black}{For ease of exposition, all the proofs are placed in the appendix.}}

\section*{Notation}
In what follows, we consider a compact $d$-dimensional submanifold without boundary $M \subset \R^D$ to be reconstructed.
For all $p \in M$, $T_p M$ designates the tangent space of $M$ at $p$. Tangent spaces will either be considered vectorial or affine depending on the context.
The standard inner product in $\R^D$ is denoted by $\langle \cdot,\cdot \rangle$ and the Euclidean distance $\norm{\cdot}$. We let $\mathcal{B}(p,r)$ denote the closed Euclidean ball of radius $r>0$ centered at $p$. We let $\wedge$ and $\vee$ denote respectively the minimum and the maximum of real numbers.
As introduced in \cite{federer1959}, the reach of $M$, denoted by $\mathrm{reach}(M)$ is the maximal offset radius for which the projection $\pi_M$ onto $M$ is well defined.
Denoting by $\dd(\cdot,M)$ the distance to $M$, the \textit{medial axis} of $M$ $\mathrm{med}(M) = \{x \in \R^D | \exists a \neq b \in M, \norm{x-a} = \norm{x-b} = \dd(x,M) \}$ is the set of points which have at least two nearest neighbors on $M$. Then, $\mathrm{reach}(M) = \underset{p \in M}{\inf} \dd(p,\mathrm{med}(M))$. We simply write $\pi$ for $\pi_M$ when there is no possibility of confusion.
For any smooth function $\Phi: \R^D \rightarrow \R^D$, we let $\dd_a \Phi$ and $\dd_a^2 \Phi$ denote the first and second order differentials of $\Phi$ at $a\in \R^D$.
For a linear map $A$, $A^t$ designates its transpose. Let $\norm{A}_\mathrm{op}= {\sup_x} \frac{\norm{A x}}{\norm{x}}$ and $\norm{A}_\mathcal{F} = \sqrt{\mathrm{trace}\left(A^t A\right)}$ denote respectively the operator norm induced by the Euclidean norm and the Frobenius norm. 
The distance between two linear subspaces $U,V \subset \R^D$ of the same dimension is measured by the sine {\color{black}{$\angle(U,V) = \underset{u\in U}{\max} ~ \underset{v' \in V^{\perp}}{\max} \dfrac{\langle u,v' \rangle}{\norm{u} \norm{v'}} = \norm{\pi_U - \pi_V}_\mathrm{op}$}} of their largest principal angle. 
The Hausdorff distance between two compact subsets $K,K'$ of $\R^D$ is denoted by $\dd_{H}(K,K') = \sup_{x \in \R^D} \left|d(x,K) - d(x,K') \right|$. Finally, we let $\cong$ denote the ambient isotopy relation in $\R^D$.

\noindent Throughout this paper, $C_\alpha$ will denote a generic constant depending on the parameter $\alpha$. For clarity's sake, $c_\alpha$ and $K_\alpha$ may also be used when several constants are involved.

\section{Minimax Risk and Main Results}
\label{Minimax_Risk_and_Optimality_section}
\subsection{Statistical Model}
\label{statistical_model_section}
Let us describe the general statistical setting we will use to define optimality for manifold reconstruction. A \textit{statistical model} $\mathcal{D}$ is a set of probability distributions on $\R^D$. In any statistical experiment, $\mathcal{D}$ is fixed and known. 
We observe an independent and identically distributed sample of size $n$ (or i.i.d. $n$-sample) $\X_n = \left\{X_1,\ldots,X_n\right\}$ drawn according to some unknown distribution $P \in \mathcal{D}$. If no noise is allowed, the problem is to recover the \textit{support} of $P$, that is, the smallest closed set $C\subset \R^D$ such that $P(C)=1$.
Let us give two examples of such models $\mathcal{D}$ by describing those of interest in this paper.

Let $\mathcal{M}_{D,d,\rho}$ be the set of all compact $d$-dimensional connected submanifolds $M \subset \R^D$ without boundary satisfying $\mathrm{reach}(M)\geq \rho$. The reach assumption is crucial to avoid arbitrarily curved and pinched shapes~\cite{Cuevas2004}. From a reconstruction point of view, $\rho$ gives a minimal feature size on $M$, and then a minimal scale for geometric information.
Every $M\in \mathcal{M}_{D,d,\rho}$ inherits a measure induced by the $d$-dimensional Hausdorff measure on $\R^D\supset M$. We denote this induced measure by $v_M$.
Beyond the geometric restrictions induced by the lower bound $\rho$ on the reach, it also requires the natural measure $v_M$ to behave like a $d$-dimensional measure, up to uniform constants. 
Denote by $\mathcal{U}_M(f_{min},f_{max})$ the set of probability distributions $Q$ having a density $f$ with respect to $v_M$ such that $0< f_{min} \leq f(x) \leq f_{max}< \infty$ for all $x\in M$.
In particular, notice that such distributions $Q \in \mathcal{U}_M(f_{min},f_{max})$ all have support $M$.  
Roughly speaking, when $Q \in \mathcal{U}_M(f_{min},f_{max})$, points are drawn almost uniformly on $M$. 
This is to ensure that the sample visits all the areas of $M$ with high probability. 
The noise-free model $\mathcal{G}_{D,d,f_{min},f_{max},\rho}$ consists of the set of all these almost uniform measures on submanifolds of dimension $d$ having reach greater than a fixed value $\rho >0$.
\begin{Def}[Noise-free model]\label{nonoise_model_definition}
$
\mathcal{G}_{D,d,f_{min},f_{max},\rho} = \bigcup_{M \in \mathcal{M}_{D,d,\rho}} \mathcal{U}_M(f_{min},f_{max})
$.
\end{Def}
Notice that we do not explicitly impose a bound on the diameter of $M$. Actually, a bound is implicitly present in the model, as stated in the next lemma, the proof of which follows from a volume argument.
\begin{lem}\label{model_properties}
There exists $C_d>0$ such that for all $Q \in \mathcal{G}_{D,d,f_{min},f_{max},\rho}$ with associated $M$,
\[
\mathrm{diam}(M) \leq \frac{C_d}{\rho^{d-1} f_{min}} =: K_{d,f_{min},\rho}.
\]
\end{lem}

Observed random variables with distribution belonging to the noise-free model $\mathcal{G}_{D,d,f_{min},f_{max},\rho}$ lie exactly on the submanifold of interest $M$. A more realistic model should allow some measurement error, as illustrated by Figure \ref{mixture1}. We formalize this idea with the following additive noise model.

\begin{Def}[Additive noise model]\label{tubular_model_definition}
For $\sigma < \rho$, we let $\mathcal{G}_{D,d,f_{min},f_{max},\rho,\sigma}$ denote the set of distributions of random variables $X = Y + Z$, where $Y$ has distribution $Q \in \mathcal{G}_{D,d,f_{min},f_{max},\rho}$, and $\norm{Z} \leq \sigma$ almost surely. 
\end{Def}

Let us emphasize that we do not require $Y$ and $Z$ to be independent, nor $Z$ to be orthogonal to $T_Y M$, as done for the ``perpendicular'' noise model of \cite{Niyogi08,Genovese12}.
This model is also slightly more general than the one considered in \cite{Maggioni16}.
Notice that the noise-free model can be thought of as a particular instance of the additive noise model, since $\mathcal{G}_{D,d,f_{min},f_{max},\rho} = \mathcal{G}_{D,d,f_{min},f_{max},\rho,\sigma=0}$.

Eventually, we may include distributions contaminated with outliers uniformly drawn in a ball $\mathcal{B}_0$ containing $M$, as illustrated in Figure \ref{mixture2}. Up to translation, we can always assume that $M \ni 0$.
To avoid boundary effects, $\mathcal{B}_0$ will be taken to contain $M$ amply, so that the outlier distribution surrounds $M$ everywhere. Since $M$ has at most diameter $K_{d,f_{min},\rho}$ from Lemma \ref{model_properties} we arbitrarily fix $\mathcal{B}_0 = \mathcal{B}(0,K_0)$, where $K_0 = K_{d,f_{min},\rho} + \rho$. Notice that the larger the radius of $\mathcal{B}_0$, the easier to label the outlier points since they should be very far away from each other.

\begin{figure}[!ht]
    \subfloat[Circle with noise:\:$d=1$, $D=2$, $\sigma>0$.\label{circle_noise}]{
      {\includegraphics[width=0.45\textwidth]{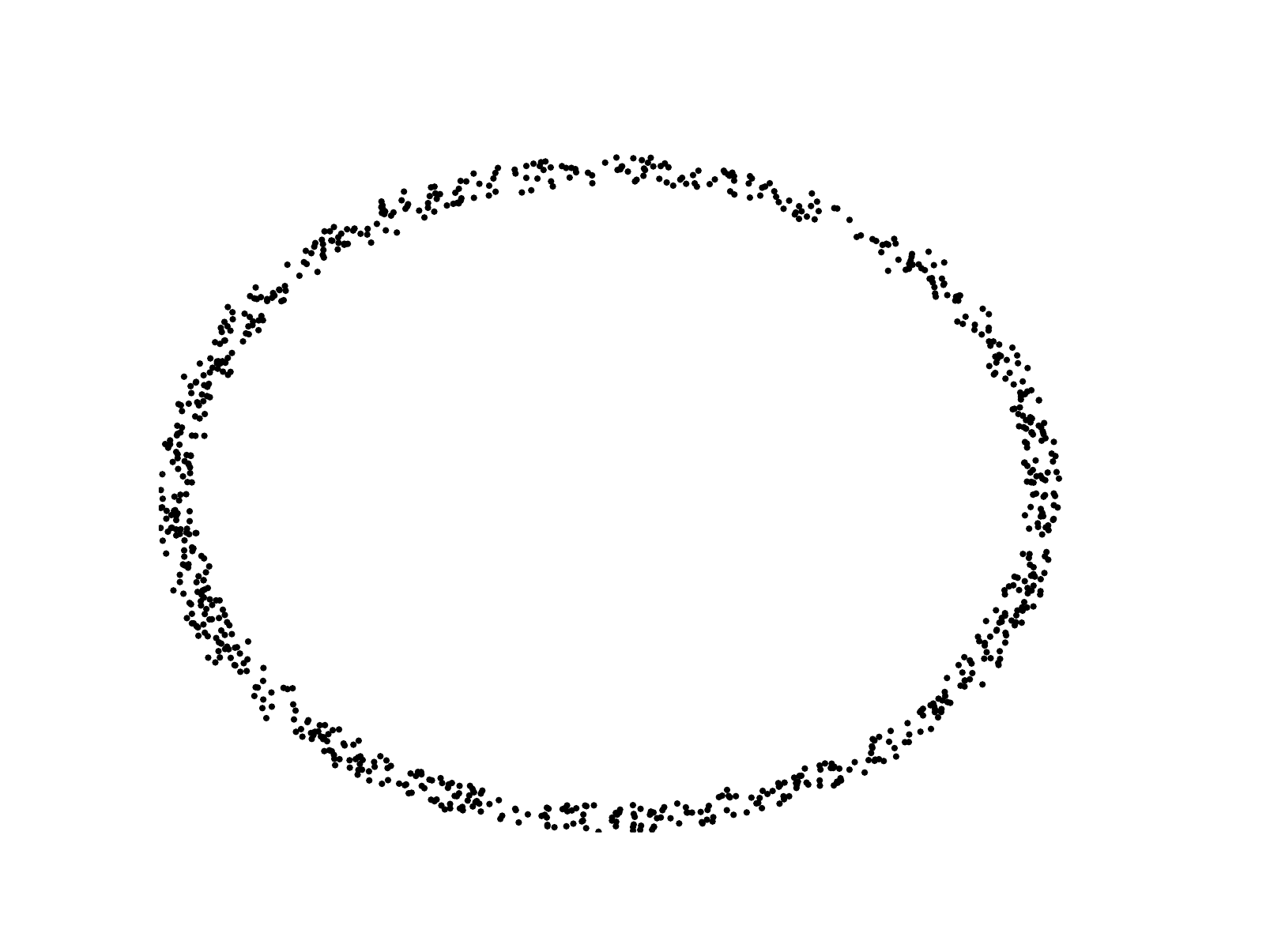}
    }
    \label{mixture1}
    }
    \hfill
    \subfloat[Torus with outliers:\:$d=2$, $D=3$, $\beta < 1$.\label{torus_noise}]{
      \includegraphics[width=0.45\textwidth]{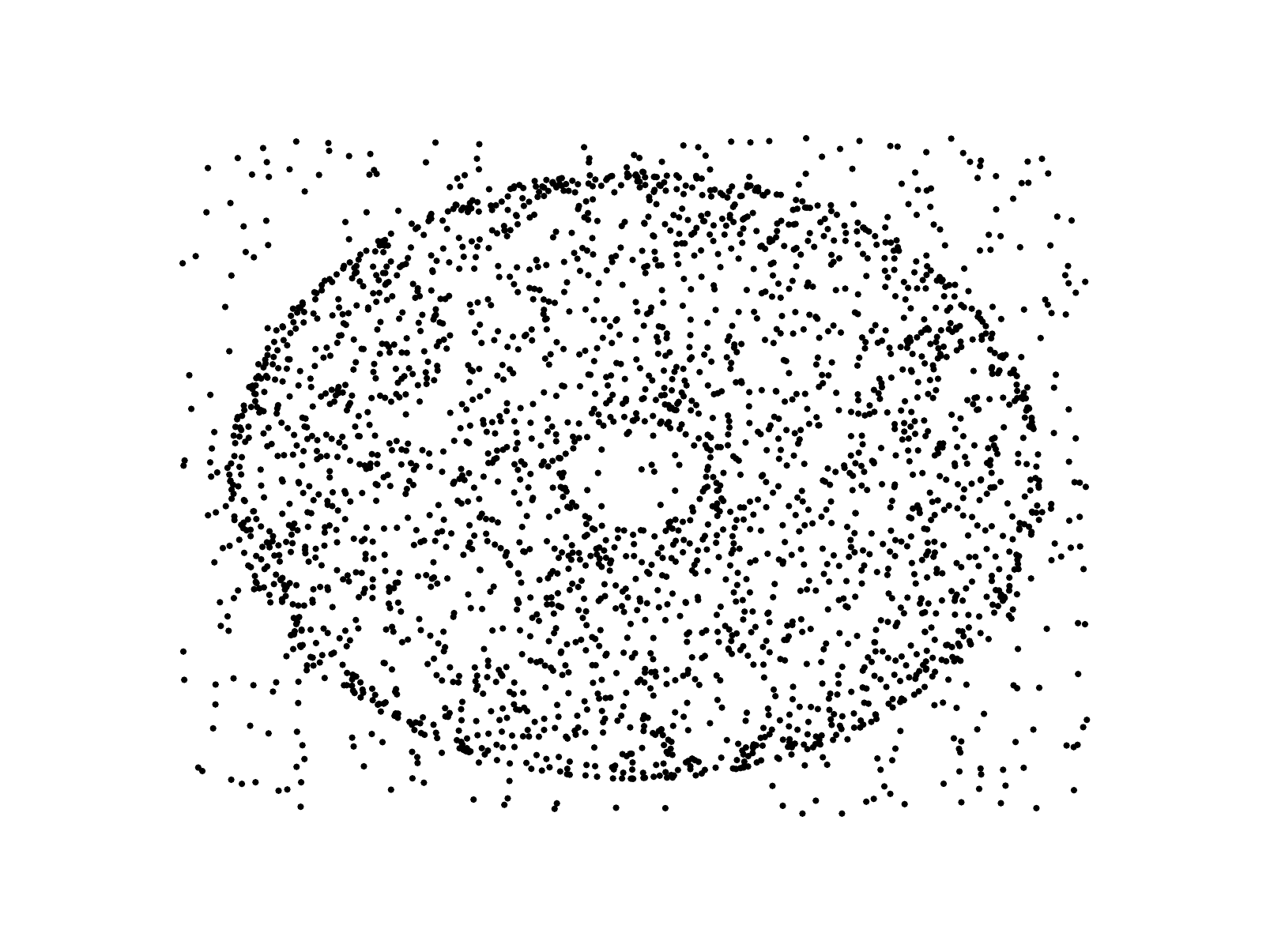}
	\label{mixture2}

    }
    \caption{Point clouds $\X_n$ drawn from distributions in $\mathcal{G}_{D,d,f_{min},f_{max},\rho,\sigma}$ (left) and $\mathcal{O}_{D,d,f_{min},f_{max},\rho,\beta}$ (right).}
	\label{mixture}
  \end{figure}

\begin{Def}[Model with outliers/Clutter noise model]\label{noise_model_definition}
For $0<f_{min} \leq f_{max} < \infty$, $0<\beta\leq 1$, and $\rho>0$, we define $\mathcal{O}_{D,d,f_{min},f_{max},\rho,\beta}$ to be the set of mixture distributions
\[
P = \beta Q + (1-\beta)U_{\mathcal{B}_0},
\]
where $Q \in \mathcal{G}_{D,d,f_{min},f_{max},\rho}$ has support $M$ such that $0 \in M$, and  $U_{\mathcal{B}_0}$ is the uniform distribution on $\mathcal{B}_0 = \mathcal{B}(0,K_0)$.
\end{Def}
\label{labels}Alternatively, a random variable $X$ with distribution $P \in \mathcal{O}_{D,d,f_{min},f_{max},\rho,\beta}$ can be represented as $X = V X' + (1-V)X''$ , where $V \in \left\{0,1\right\}$ is a Bernoulli random variable with parameter $\beta$, $X'$ has distribution in $\mathcal{G}_{D,d,f_{min},f_{max},\rho}$ and $X''$ has a uniform distribution over $\mathcal{B}_0$, and such that $V,X',X''$ are independent.
In particular for $\beta = 1$, $ \mathcal{O}_{D,d,f_{min},f_{max},\rho,\beta=1} = \mathcal{G}_{D,d,f_{min},f_{max},\rho}$.

\subsection{Minimax Risk}
\label{subsec:minimax_risk}
For a probability measure $P \in \mathcal{D}$, denote by $\E_P$  --- or simply $\E$ --- the expectation with respect to the product measure $P^{(n)}$. The quantity we will be interested in is the \textit{minimax risk} associated to the model $\mathcal{D}$. For $n\geq 0$,
\[ R_n(\mathcal{D}) = \inf_{\hat{M}} \sup_{P \in \mathcal{D}} \E_P \left[ \dd_{{H}}\left(M,\hat{M}\right)\right], \]
where the {infimum} is taken over all the estimators $\hat{M} = \hat{M}\left(X_1,\ldots,X_n\right)$ computed over an $n$-sample. $R_n(\mathcal{D})$ is the best risk that an estimator based on an $n$-sample can achieve uniformly over the class $\mathcal{D}$. It is clear from the definition that if $\mathcal{D}' \subset \mathcal{D}$ then $R_n(\mathcal{D'}) \leq R_n(\mathcal{D})$. It follows the intuition that the broader the class of considered manifolds, the more difficult it is to estimate them uniformly well. Studying $R_n(\mathcal{D})$ for a fixed $n$ is a difficult task that can rarely be carried out. We will focus on the semi-asymptotic behavior of this risk. 
As $R_n(\mathcal{D})$ cannot be surpassed, its rate of convergence to $0$ as $n\rightarrow \infty$ may be seen as the best rate of approximation that an estimator can achieve.
We will say that two sequences $(a_n)_n$ and $(b_n)_n$ are asymptotically comparable, denoted by $a_n \asymp b_n$, if there exist $c,C>0$ such that for $n$ large enough, $cb_n \leq a_n \leq Cb_n$.
\begin{Def}
\label{minimax_def}
An estimator $\hat{M}$ is said to be \textit{(asymptotically) minimax optimal} over $\mathcal{D}$ if
\[
\sup_{P \in \mathcal{D}} \E_P \left[ \dd_{{H}}\left(M,\hat{M}\right)\right]
\asymp
R_n(\mathcal{D}).
\]
\end{Def}
In other words, $\hat{M}$ is (asymptotically) minimax optimal if it achieves, up to constants, the best possible rate of convergence in the worst case.

Studying a minimax rate of convergence is twofold. On one hand, deriving an upper bound on $R_n$ boils down to provide an estimator and to study its quality uniformly on $\mathcal{D}$. On the other hand, bounding $R_n$ from below amounts to study the worst possible case in $\mathcal{D}$. This part is usually achieved with standard Bayesian techniques \cite{Lecam73}. For the models considered in the present paper, the rates were given in \cite{Genovese12,Kim2015}.

\begin{thm}[Theorem 3 of \cite{Kim2015}]\label{minimax_rates}
We have,
\begin{align}
R_n\left(\mathcal{G}_{D,d,f_{min},f_{max},\rho}\right)
\asymp
\left(\frac{\log n}{ n}\right)^{2/d},
\tag{Noise-free}
\end{align}
and for $0 < \beta\leq 1$ fixed,
\begin{align*}
R_n\left(\mathcal{O}_{D,d,f_{min},f_{max},\rho,\beta}\right) 
	\asymp \left(\frac{\log n}{\beta n}\right)^{2/d}.
	\tag{Clutter noise}
\end{align*}
\end{thm}
Since the additive noise model $\mathcal{G}_{D,d,f_{min},f_{max},\rho,\sigma}$ has not yet been considered in the literature, the behavior of the associated minimax risk is not known.
Beyond this theoretical result, an interesting question is to know whether these minimax rates can be achieved by a tractable algorithm. Indeed, that proposed in~\cite{Genovese12} especially rely on a minimization problem over the class of submanifolds $\mathcal{M}_{D,d,\rho}$, which is computationally costly.
In addition, the proposed estimators are themselves submanifolds, which raises storage problems.
Moreover, no guarantee is given on the topology of the estimators. 
Throughout the present paper, we will build estimators that address these issues.

\subsection{Main Results}
\label{subsec:main_results}
Let us start with the additive noise model $\mathcal{G}_{D,d,f_{min},f_{max},\rho,\sigma}$, that includes in particular the noise-free case $\sigma = 0$. 
The estimator $\hat{M}_{\mathtt{TDC}}$  is based on the Tangential Delaunay Complex (Section \ref{tangential_delaunay_complex_section}), with a tangent space estimation using a local PCA (Section \ref{tangent_space_estimation_and_denoising_procedure}).

\begin{thm}
\label{main_result_noob_nonoise}
$\hat{M}_{\mathtt{TDC}} = \hat{M}_{\mathtt{TDC}}(\X_n)$ is a simplicial complex with vertices included in $\X_n$ such that the following holds.
There exists $\lambda_{d,f_{min},f_{max}}> 0$ such that if $\sigma \leq \lambda \left(\frac{\log n}{n} \right)^{1/d}$ with $\lambda \leq \lambda_{d,f_{min},f_{max}}$, then
\[
\underset{n \rightarrow \infty}{\lim} 
\p \left(
	\dd_{H}({M},\hat{M}_{\mathtt{TDC}}) \leq {C_{d,f_{min},f_{max},{\rho}}} \left\lbrace \left( \frac{\log n}{n} \right)^{2/d} \vee \lambda^2 \right\rbrace
	\text{ and }
	{M} \cong \hat{M}_{\mathtt{TDC}}
 \right)
 = 1.
\]
Moreover, for $n$ large enough,
\[
\sup_{Q \in \mathcal{G}_{D,d,f_{min},f_{max},\rho,\sigma}} \E_Q  \dd_{H}({M},\hat{M}_{\mathtt{TDC}}) \leq {C'_{d,f_{min},f_{max},{\rho}}} \left\lbrace \left( \frac{\log n}{n} \right)^{2/d} \vee \lambda^2 \right\rbrace.
\]
\end{thm}
It is interesting to note that the constants appearing in Theorem \ref{main_result_noob_nonoise} do not depend on the ambient dimension $D$.
Since $R_n\left(\mathcal{G}_{D,d,f_{min},f_{max},\rho,\sigma}\right) \geq R_n\left(\mathcal{G}_{D,d,f_{min},f_{max},\rho}\right)$, we obtain immediately from  Theorem \ref{main_result_noob_nonoise} that $\hat{M}_{\mathtt{TDC}}$ achieves the minimax optimal rate $\left({\log n}/{n} \right)^{2/d}$ over $\mathcal{G}_{D,d,f_{min},f_{max},\rho,\sigma}$ when $\sigma \leq c_{d,f_{min},f_{max}} \left({\log n}/{n} \right)^{2/d}$.
Note that the estimator of \cite{Maggioni16} achieves the rate $\left({\log n}/{n} \right)^{2/(d+2)}$ when $\sigma \leq c_{d,f_{min},f_{max}} \left({\log n}/{n} \right)^{2/(d+2)}$, so does the estimator of \cite{Genovese11} for $\sigma < \rho$ if the noise is centered and perpendicular to the submanifold. 
As a consequence, $\hat{M}_{\mathtt{TDC}}$ outperforms these two existing procedures whenever  $\sigma \ll \left( \log n / n \right)^{2/(d+2)}$, with the additional feature of exact topology recovery.
Still, for $\sigma \gg  \left( \log n / n \right)^{1/d}$, $\hat{M}_{\mathtt{TDC}}$ may perform poorly compared to \cite{Genovese11}. This might be due to the fact that the vertices of $\hat{M}_{\mathtt{TDC}}$ are sample points themselves, while for higher noise levels, a pre-process of the data based on local averaging could be more relevant.

In the model with outliers $\mathcal{O}_{D,d,f_{min},f_{max},\rho,\beta}$, with the same procedure used to derive Theorem \ref{main_result_noob_nonoise} and an additional iterative preprocessing of the data based on local PCA to remove outliers (Section \ref{tangent_space_estimation_and_denoising_procedure}), we design an estimator of $M$ that achieves a rate as close as wanted to the noise-free rate. Namely, for any positive $\delta < 1/(d(d+1))$, we build $\hat{M}_{\mathtt{TDC \delta}}$ that satisfies the following similar statement.

\begin{thm}
\label{main_result_noob_noise}
$\hat{M}_{\mathtt{TDC}\delta} = \hat{M}_{\mathtt{TDC}\delta}(\X_n)$ is a simplicial complex with vertices included in $\X_n$ such that
\[
\underset{n \rightarrow \infty}{\lim} 
\p \left(
	\dd_{H}({M},\hat{M}_{\mathtt{TDC \delta}}) \leq C_{d,f_{min},f_{max},\rho}  \left( \dfrac{\log n}{\beta n} \right)^{2/d - 2 \delta}
	\text{ and }
	{M} \cong \hat{M}_{\mathtt{TDC \delta}}
 \right)
 = 1.
\]
Moreover, for $n$ large enough,
\[
\sup_{P \in \mathcal{O}_{D,d,f_{min},f_{max},\rho,\beta}} \E_P  \dd_{H}({M},\hat{M}_{\mathtt{TDC \delta}}) \leq C'_{d,f_{min},f_{max},\rho}  \left( \dfrac{\log n}{\beta n} \right)^{2/d - 2 \delta}.
\]
\end{thm}

$\hat{M}_{\mathtt{TDC \delta}}$ converges at the rate at least $\left( {\log n}/{n} \right)^{2/d - 2\delta}$, which is not the minimax optimal rate according to Theorem \ref{minimax_rates}, but that can be set as close as desired to it.
To our knowledge, $\hat{M}_{\mathtt{TDC\delta}}$ is the first explicit estimator to provably achieve such a rate in the presence of outliers. Again, it is worth noting that the constants involved in Theorem \ref{main_result_noob_noise} do not depend on the ambient dimension $D$.
The construction and computation of $\hat{M}_{\mathtt{TDC \delta}}$ is the same as $\hat{M}_{\mathtt{TDC}}$, with an extra pre-processing of the point cloud allowing to remove outliers. 
This decluttering procedure leads to compute, at each sample point, at most $\log(1/\delta)$ local PCA's, instead of a single one for $\hat{M}_{\mathtt{TDC}}$.

From a theoretical point of view, there exists a (random) number of iterations of this decluttering process, from which an estimator $\hat{M}_{\mathtt{TDC +}}$ can be built to satisfy the following.

\begin{thm}
\label{theoretical_rate}
$\hat{M}_{\mathtt{TDC +}} = \hat{M}_{\mathtt{TDC +}}(\X_n)$ is a simplicial complex of vertices contained in $\X_n$ such that

\[
\underset{n \rightarrow \infty}{\lim} 
\p \left(
	\dd_{H}({M},\hat{M}_{\mathtt{TDC+}}) \leq C_{d,f_{min},f_{max},\rho}  \left( \dfrac{\log n}{\beta n} \right)^{2/d}
	\text{ and }
	{M} \cong \hat{M}_{\mathtt{TDC+}}
 \right)
 = 1.
\]
Moreover, for $n$ large enough,
\[
\sup_{P \in \mathcal{O}_{D,d,f_{min},f_{max},\rho,\beta}} \E_P  \dd_{H}({M},\hat{M}_{\mathtt{TDC+}}) \leq C'_{d,f_{min},f_{max},\rho}  \left( \dfrac{\log n}{\beta n} \right)^{2/d}.
\]
\end{thm}

$\hat{M}_{\mathtt{TDC +}}$ may be thought of as a limit of $\hat{M}_{\mathtt{TDC \delta}}$ when $\delta$ goes to $0$. As it will be proved in Section \ref{tangent_space_estimation_and_denoising_procedure}, this limit will be reached for $\delta$ close enough to $0$. Unfortunately this convergence threshold is also random, hence unknown.

The statistical analysis of the reconstruction problem is postponed to Section \ref{tangent_space_estimation_and_denoising_procedure}.
Beforehand, let us describe the Tangential Delaunay Complex in a deterministic and idealized framework where the tangent spaces are known and no outliers are present.

\section{Tangential Delaunay Complex}
\label{tangential_delaunay_complex_section}
Let $\mathcal{P}$ be a finite subset of $\R^D$. In this section, we denote the point cloud $\mathcal{P}$ to emphasize the fact that it is considered nonrandom. For $\varepsilon,\delta >0$, $\mathcal{P}$ is said to be $\varepsilon$\textit{-dense} in $M$ if $\sup_{x \in M}d(x,\mathcal{P}) \leq \varepsilon$, and $\delta$\textit{-sparse} if $\dd(p,\mathcal{P} \setminus\{p\}) \geq \delta$ for all $p \in \mathcal{P}$.
A $(\delta,\varepsilon)$\textit{-net} (of $M$) is a $\delta$-sparse and $\varepsilon$-dense point cloud.

\subsection{Restricted Weighted Delaunay Triangulations}
We now assume that $\mathcal{P} \subset M$.
A weight assignment to $\mathcal{P}$ is a function $\omega : \mathcal{P} \longrightarrow [0,\infty)$. The \textit{weighted Voronoi diagram} is defined to be the Voronoi diagram associated to the weighted distance $\dd(x,p^\omega)^2 = \norm{x-p}^2 - \omega(p)^2$. Every $p \in \mathcal{P}$ is associated to its weighted Voronoi cell $ \mathrm{Vor}^\omega(p)$.
For $\tau \subset \mathcal{P}$, let
\[ \mathrm{Vor}^\omega(\tau) = \bigcap_{p\in\tau} \mathrm{Vor}^\omega(p)
\]
be the common face of the weighted Voronoi cells of the points of $\tau$. The \textit{weighted Delaunay triangulation} $\mathrm{Del}^\omega(\mathcal{P})$ is the dual triangulation to the decomposition given by the weighted Voronoi diagram. In other words, for $\tau \subset \mathcal{P}$, the simplex with vertices $\tau$, also denoted by $\tau$, satisfies
\[
\tau \in \mathrm{Del}^\omega(\mathcal{P})
\Leftrightarrow
\mathrm{Vor}^\omega(\tau) \neq \emptyset.
\]
Note that for a constant weight assignment $\omega(p) \equiv \omega_0$, $\mathrm{Del}^\omega(\mathcal{P})$ is the usual Delaunay triangulation of $\mathcal{P}$.
Under genericity assumptions on $\mathcal{P}$ and bounds on $\omega$,  $\mathrm{Del}^\omega(\mathcal{P})$ is an embedded triangulation with vertex set $\mathcal{P}$ \cite{Boissonnat14}. 
The reconstruction method proposed in this paper is based on $\mathrm{Del}^\omega(\mathcal{P})$ for some weights $\omega$ to be chosen later.
As it is a triangulation of the whole convex hull of $\mathcal{P}$ and fails to recover the geometric structure of $M$, we take restrictions of it in the following manner. 

Given a family $R = \left\{R_p\right\}_{p\in \mathcal{P}}$ of subsets $R_p \subset \R^D$ indexed by $\mathcal{P}$, the weighted Delaunay complex restricted to $R$ is the sub-complex of $\mathrm{Del}^\omega(\mathcal{P})$ defined by
\[
\tau \in \mathrm{Del}^\omega(\mathcal{P},R)
\Leftrightarrow
\mathrm{Vor}^\omega(\tau) \cap \left( \bigcup_{p \in \tau} R_p \right) \neq \emptyset.
\]
In particular, we define the \textit{Tangential Delaunay Complex} $\mathrm{Del}^\omega(\mathcal{P},T)$ by taking $R = T = \left\{T_p M \right\}_{p\in \mathcal{P}}$, the family of tangent spaces taken at the points of $\mathcal{P} \subset M$ \cite{Boissonnat14}. $\mathrm{Del}^\omega(\mathcal{P},T)$ is a pruned version of $\mathrm{Del}^\omega(\mathcal{P})$ where only the simplices with directions close to the tangent spaces are kept. 
{\color{black}{Indeed, $T_p M$ being the best linear approximation of $M$ at $p$, it is very unlikely for a reconstruction of $M$ to have components in directions normal to $T_p M$ (see Figure \ref{tangential_illustration}).}}
As pointed out in \cite{Boissonnat14}, computing $\mathrm{Del}^\omega(\mathcal{P},T)$ only requires to compute Delaunay triangulations in the tangent spaces that have dimension $d$. This reduces the computational complexity dependency on the ambient dimension $D > d$.
\begin{figure}[ht]
\centering
\includegraphics[width=0.8\textwidth]{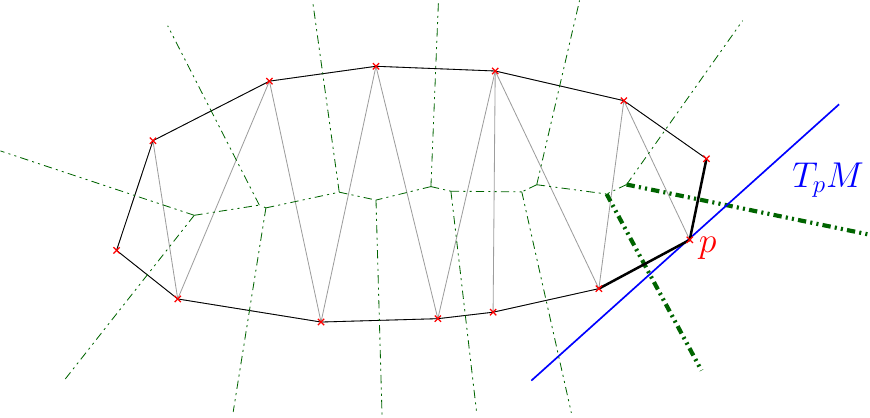}

\caption{Construction of $\mathrm{Del}^\omega(\mathcal{P},T)$ at $p$ for $\omega \equiv 0$: $p$ has three incident edges in the ambient Delaunay triangulation, but only two (bold) have dual Voronoi face intersecting $T_pM$.}
\label{tangential_illustration}
\end{figure}
The weight assignment $\omega$ gives degrees of freedom for the reconstruction. The extra degree of freedom $\omega$ permits to stabilize the triangulation and to remove the so-called \textit{inconsistencies}, the points remaining fixed. For further details, see \cite{Boissonnat09,Boissonnat14}.
\subsection{Guarantees}

The following result sums up the reconstruction properties of the Tangential Delaunay Complex that we will use. For more details about it, the reader is referred to \cite{Boissonnat14}.
\begin{thm}[Theorem 5.3 in \cite{Boissonnat14}]
\label{Resultats_JDB}
There exists $\varepsilon_d>0$ such that for all $\varepsilon \leq \varepsilon_d \rho$ and all $M \in \mathcal{M}_{D,d,\rho}$, if $\mathcal{P} \subset M$ is an $( \varepsilon,2\varepsilon)$-net, there exists a weight assignment $\omega_\ast = \omega_{\ast \mathcal{P},T}$ depending on $\mathcal{P}$ and $T = \left\{T_p M\right\}_{p\in \mathcal{P}}$ such that
\begin{itemize}
\item $\dd_{H} \left( M, \mathrm{Del}^{\omega_\ast}(\mathcal{P},T) \right) \leq {C_{d}} {\varepsilon^2}/{\rho}$,
\item $M$ and $\mathrm{Del}^{\omega_\ast}(\mathcal{P},T)$ are ambient isotopic.
\end{itemize}
\end{thm}
Computing $\mathrm{Del}^{\omega_\ast}(\mathcal{P},T)$ requires to determine the weight function $\omega_\ast = \omega_{\ast \mathcal{P},T}$. In \cite{Boissonnat14}, a greedy algorithm is designed for this purpose and has a time complexity ${O}\left( Dn^2 + D 2^{O(d^2)}n \right)$.

Given an $(\varepsilon,2\varepsilon)$-net $\mathcal{P}$ for $\varepsilon$ small enough, $\mathrm{Del}^{\omega_\ast}(\mathcal{P},T)$ recovers $M$ up to ambient isotopy and approximates it at the scale $\varepsilon^2$. The order of magnitude $\varepsilon^2$ with an input $\mathcal{P}$ of scale $\varepsilon$ is remarkable. Another instance of this phenomenon is present in \cite{Clarkson2006} in codimension $1$. 
We will show that this $\varepsilon^2$ provides the minimax rate of approximation when dealing with random samples. Therefore, it can be thought of as optimal.

Theorem \ref{Resultats_JDB} suffers two major imperfections. First, it requires the knowledge of the tangent spaces at each sample point  --- since ${\omega_\ast} = {\omega_{\ast \mathcal{P}, T}}$ --- and it is no longer usable if tangent spaces are only known up to some error. Second, the points are assumed to lie exactly on the submanifold $M$, and no noise is allowed.
The analysis of $\mathrm{Del}^{\omega_\ast}(\mathcal{P},T)$ is sophisticated \cite{Boissonnat14}.
Rather than redo the whole study with milder assumptions, we tackle this question with an approximation theory approach (Theorem \ref{tangential_perturbation_thm}). Instead of studying if $\mathrm{Del}^{\omega_\ast}(\mathcal{P}',T')$ is stable
when $\mathcal{P}'$ lies close to $M$ and $T'$ close to $T$,
we examine what $\mathrm{Del}^{\omega_\ast}(\mathcal{P}',T')$ actually reconstructs, as detailed in Section \ref{stability_section}.

\subsection{On the Sparsity Assumption}
\label{sparsity_section}

In Theorem \ref{Resultats_JDB}, $\mathcal{P}$ is assumed to be dense enough so that it covers all the areas of $M$. It is also supposed to be sparse at the same scale as the density parameter $\varepsilon$.  Indeed, arbitrarily accumulated points would generate non-uniformity and instability for $\mathrm{Del}^{\omega_\ast}(\mathcal{P},T)$ \cite{Boissonnat09,Boissonnat14}. At this stage, we emphasize that the construction of a $(\varepsilon,2\varepsilon)$-net can be carried out given an $\varepsilon$-dense sample. Given an $\varepsilon$-dense sample $\mathcal{P}$, the \textit{farthest point sampling} algorithm prunes $\mathcal{P}$ and outputs an $(\varepsilon,2 \varepsilon)$-net $\mathcal{Q} \subset \mathcal{P}$ of $M$ as follows. Initialize at $\mathcal{Q} = \{p_1\} \subset \mathcal{P}$, and while $\underset{p \in \mathcal{P}}{\max} ~ \dd(p,\mathcal{Q}) > \varepsilon$, add to $\mathcal{Q}$ the farthest point to $\mathcal{Q}$ in $\mathcal{P}$, that is, $\mathcal{Q} \leftarrow \mathcal{Q} \cup  \{ \underset{p \in \mathcal{P}}{\argmax} ~ \dd(p,\mathcal{Q}) \}$.
The output $\mathcal{Q}$ is $\varepsilon$-sparse and satisfies $\dd_{H}(\mathcal{P},\mathcal{Q})\leq \varepsilon$, so it is a $(\varepsilon,2\varepsilon)$-net of $M$. Therefore, up to the multiplicative constant $2$, sparsifying $\mathcal{P}$ at scale $\varepsilon$ will not deteriorate its density property. Then, we can run the farthest point sampling algorithm to preprocess the data, so that the obtained point cloud is a net.

\label{farthest_point_sampling}

\section{Stability Result}
\label{stability_section}
\subsection{Interpolation Theorem}
\label{interpolation_theorem_section}

As mentioned above, if the data do not lie exactly on $M$ and if we do not have the exact knowledge of the tangent spaces, Theorem \ref{Resultats_JDB} does not apply. To bypass this issue, we interpolate the data with another submanifold $M'$ satisfying good properties, as stated in the following result.
\begin{thm}[Interpolation]
\label{tangential_perturbation_thm}
Let $M \in \mathcal{M}_{D,d,\rho}$. Let $\mathcal{P} = \{p_1,\ldots,p_q\} \subset \R^D$  be a finite point cloud and $\tilde{T} = \left\{ \tilde{T}_1,\ldots,\tilde{T}_q \right\}$ be a family of $d$-dimensional linear subspaces of $\R^D$. For $\theta \leq   \pi/64$ and $18 \eta <\delta \leq \rho $, assume that
	\begin{itemize}
	\item $\mathcal{P}$ is $\delta$-sparse: $\underset{i\neq j}{\min} \norm{p_j - p_i} \geq  \delta$,
	\item the $p_j$'s are $\eta$-close to $M$: $\underset{1 \leq j \leq q}{\max} \dd(p_j,M) \leq \eta$,
	\item $\underset{1 \leq j \leq q}{\max} \: \angle(T_{\pi_M(p_j)} M, \tilde{T}_j) \leq \sin \theta$.
	\end{itemize}
	 Then, there exist a universal constant $c_0 \leq 285$ 
	 and a compact $d$-dimensional connected submanifold $M' \subset \mathbb{R}^D$ without boundary such that
	\begin{multicols}{2}
	\begin{enumerate}
	\item $\mathcal{P} \subset M'$,
	\item $\mathrm{reach}\bigl(M'\bigr) \geq \left( 1- c_0 \left(\frac{\eta}{\delta} +{\theta}\right) \frac{\rho}{\delta} \right) \rho,$
	\item $ T_{p_j} M' = \tilde{T}_j$ for all $1 \leq j \leq q$,
	\item $\dd_{H}(M,M') \leq \delta \theta + \eta$,
	\item $M$ and $M'$ are ambient isotopic.
	\end{enumerate}
	\end{multicols}	
\end{thm}
\begin{figure}[ht]
\centering
\includegraphics[clip = true, trim = 10mm 20mm 10mm 5mm, width=0.75\textwidth]{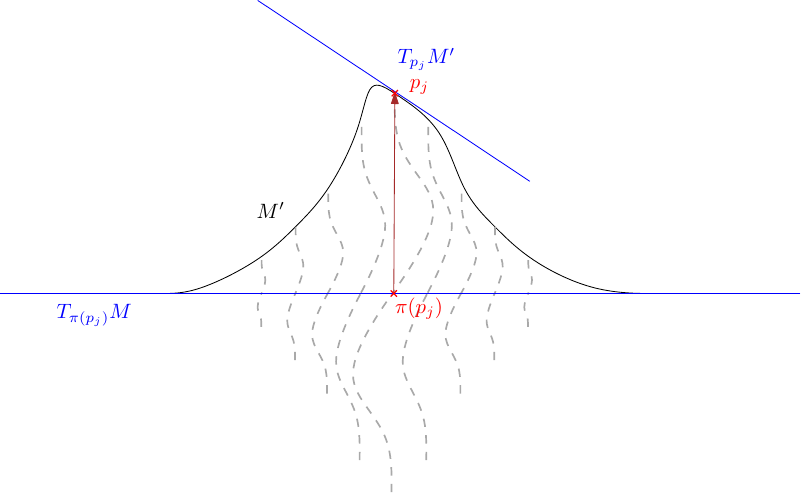}
\caption{An instance of the interpolating submanifold $M'$. Dashed lines correspond to the image of vertical lines by the ambient diffeomorphism $ \Phi $ defining $M' =  \Phi (M)$.}
\label{fig_bump}
\end{figure}
Theorem \ref{tangential_perturbation_thm} fits a submanifold $M'$ to noisy points and perturbed tangent spaces with no change of topology and a controlled reach loss. We will use $M'$ as a proxy for $M$. {\color{black}{Indeed, if $\tilde{T}_1,\ldots,\tilde{T}_q$ are estimated tangent spaces at the noisy base points $p_1,\ldots,p_q$, $M'$ has the virtue of being reconstructed by $\mathrm{Del}^{\omega_\ast}(\mathcal{P},\tilde{T})$ from Theorem \ref{Resultats_JDB}.}} Since $M'$ is topologically and geometrically close to $M$, we conclude that $M$ is reconstructed as well by transitivity. In other words, Theorem \ref{tangential_perturbation_thm} allows to consider a noisy sample with estimated tangent spaces as an exact sample with exact tangent spaces.
$M'$ is built pushing and rotating $M$ towards the $p_j$'s locally along the vector $\left(p_j - \pi(p_j\right))$, as illustrated in Figure \ref{fig_bump}. Since the construction is quite general and may be applied in various settings, let us provide an outline of the construction.

Let $\phi(x) = \exp \bigl(\frac{ \norm{x}^2}{\norm{x}^2-1} \bigr) \mathbbm{1}_{\norm{x}^2 < 1}$. $\phi$ is smooth and satisfies $\phi(0) = 1$, $\norm{\phi}_\infty \leq 1$ and $\dd_0 \phi = 0$. 
For $j = 1, \ldots, q$, it follows easily from the definition of $\angle(T_{\pi(p_j)} M, \tilde{T}_j)$ --- \textit{e.g.} by induction on the dimension --- that there exists a rotation $R_j$ of $\R^D$ mapping $T_{\pi(p_j)} M$ onto $\tilde{T}_j$ that satisfies $\norm{R_j-I_D}_{op} \leq 2 \sin \left( \theta / 2 \right) \leq \theta$. 
For $\ell > 0$ to be chosen later, and all $a \in \R^D$, let us define $ \Phi  : \R^D \rightarrow \R^D$ by
\begin{align*}
\Phi(a) &= a + \sum_{j=1}^q \phi \left(\frac{a-\pi(p_j)}{\ell} \right) \bigl[\underbrace{(R_j-I_D)(a - \pi(p_j)) + (p_j - \pi(p_j))}_{\psi_j(a)}\bigr].
\end{align*}
$\Phi $ is designed to map $\pi(p_j)$ onto $p_j$ with $\dd_{\pi(p_j)} \Phi = R_j$.
Roughly speaking, in balls of radii $\ell$ around each $\pi(p_j)$, $ \Phi $ shifts the points in the direction $p_j - \pi(p_j)$ and rotates it around $\pi(p_j)$.
Off these balls, $ \Phi $ is the identity map.
To guarantee smoothness, the shifting and the rotation are modulated by the kernel $\phi$, as $\norm{a - \pi(p_j)}$ increases.
Notice that $\dd_a \psi_j = (R_j - I_D)$ and $\norm{\psi_j(a)} \leq \ell \theta + \eta$ whenever $\phi\left(\frac{a-\pi(p_j)}{\ell} \right) \neq 0$.
Defining $M' =  \Phi (M)$, the facts that $M'$ fits to $\mathcal{P}$ and $\tilde{T}$ and is Hausdorff-close to $M$ follow by construction. Moreover, Theorem 4.19 of \cite{federer1959} (reproduced as Lemma \ref{reach_stability} in this paper) states that the reach is stable with respect to $\mathcal{C}^2$-diffeomorphisms of the ambient space. The estimate on $\mathrm{reach}(M')$ relies on the following lemma stating differentials estimates on $\Phi$.

\begin{lem}
\label{bump_lem}
There exist universal constants $C_1\leq 7/2$ and $C_2\leq 28$ such that if $ 6 \eta < \ell \leq \delta/3$ and $\theta \leq \pi/64$,  $ \Phi  : \mathbb{R}^D \longrightarrow \mathbb{R}^D$ is a global $\mathcal{C}^\infty$-diffeomorphism. In addition, for all $a$ in $\R^D$,
\begin{gather*}
 \norm{\dd_a  \Phi }_\mathrm{op} \leq 1 + C_1 \left(\frac{\eta}{\ell} + \theta \right), ~~
	\norm{\dd_a  \Phi ^{-1}}_\mathrm{op} \leq \frac{1}{1-C_1 \left(\frac{\eta}{\ell} + \theta \right)}, ~~
	\norm{\dd^2_a \Phi}_\mathrm{op} \leq C_2 \left(\frac{\eta}{\ell^2} +\frac{\theta}{\ell}\right).
\end{gather*}

\end{lem}
The ambient isotopy follows easily by considering the weighted version $ \Phi _{(t)}(a) = a + t\left(\Phi(a)-a\right)$ for $0\leq t\leq 1$ and the same differential estimates. We then take the maximum possible value $\ell = \delta/3$ and $M' = \Phi(M)$.
\begin{rem}
Changing slightly the construction of $M'$, one can also build it such that the curvature tensor at each $p_j$ corresponds to that of $M$ at $\pi(p_j)$. For this purpose it suffices to take a localizing function $\phi$ identically equal to $1$ in a neighborhood of $0$. This additional condition would impact the universal constant $c_0$ appearing in Theorem \ref{tangential_perturbation_thm}.
\end{rem}

\subsection{Stability of the Tangential Delaunay Complex}
\label{consequence_delaunay_section}
Theorem \ref{tangential_perturbation_thm} shows that even in the presence of noisy sample points at distance $\eta$ from $M$, and with the knowledge of the tangent spaces up to some angle $\theta$, it is still possible to apply Theorem \ref{Resultats_JDB} to some virtual submanifold ${M}'$. Denoting $\tilde{M} = \mathrm{Del}^{\omega_\ast}(\mathcal{P},\tilde{T})$, since $\dd_{H}(M,\tilde{M}) \leq \dd_{H}(M,M') + \dd_{H}(M',\tilde{M})$ and since the ambient isotopy relation is transitive, $M \cong M' \cong \tilde{M}$. We get the following result as a straightforward combination of Theorem \ref{Resultats_JDB} and Theorem \ref{tangential_perturbation_thm}.

\begin{thm}[Stability of the Tangential Delaunay Complex]\label{tdc_stability}
There exists $\varepsilon_{d}>0$ such that for all $\varepsilon \leq \varepsilon_{d}\rho$ and all $M \in \mathcal{M}_{D,d,\rho}$, the following holds. 
Let $\mathcal{P} \subset \R^D$ finite point cloud and $\tilde{T} = \left\{\tilde{T}_p\right\}_{p\in \mathcal{P}}$ be a family of $d$-dimensional linear subspaces of $\R^D$ such that
\begin{multicols}{2}
\begin{itemize}
\item $\underset{p \in \mathcal{P}}{\max} \: \dd(p,M) \leq \eta$,
\item $\underset{p \in \mathcal{P}}{\max} \: \angle (T_{\pi_M(p)} {M},\tilde{T}_p ) \leq \sin \theta$, 
\item $\mathcal{P}$ is $\varepsilon$-sparse,
\item $\underset{x\in M}{\max} \: \dd(x,\mathcal{P}) \leq  2\varepsilon$.
\end{itemize}
\end{multicols}
\noindent If $\theta \leq \varepsilon/(1140 \rho)$ and $\eta \leq \varepsilon^2/(1140 \rho)$, then,
\begin{itemize}
\item $\dd_{H} \left( M, \mathrm{Del}^{\omega_\ast}(\mathcal{P},\tilde{T}) \right) 
\leq 
C_{d}\varepsilon^2/\rho$,
\item $M$ and $\mathrm{Del}^{\omega_\ast}(\mathcal{P},\tilde{T})$ are ambient isotopic.
\end{itemize}
\end{thm}

Indeed, applying the reconstruction algorithm of Theorem \ref{Resultats_JDB} even in the presence of noise and uncertainty on the tangent spaces actually recovers the submanifold $M'$ built in Theorem \ref{tangential_perturbation_thm}. $M'$ is isotopic to $M$ and the quality of the approximation of $M$ is at most impacted by the term $\dd_{H}(M,M') \leq \varepsilon \theta + \eta$. 
The lower bound on $\mathrm{reach}(M')$ is crucial, as constants appearing in Theorem \ref{Resultats_JDB} are not bounded for arbitrarily small reach. 

It is worth noting that no extra analysis of the Tangential Delaunay Complex was needed to derive its stability. The argument is global, constructive, and may be applied to other reconstruction methods taking tangent spaces as input. 
For instance, a stability result similar to Theorem \ref{tdc_stability} could be derived readily for the so-called \textit{cocone} complex \cite{Dey05} using the interpolating submanifold of Theorem \ref{tangential_perturbation_thm}.

\section{Tangent Space Estimation and Decluttering Procedure}
\label{tangent_space_estimation_and_denoising_procedure}

\subsection{Additive Noise Case}
We now focus on the estimation of tangent spaces in the model with additive noise $\mathcal{G}_{D,d,f_{min},f_{max},\rho,\sigma}$. The proposed method is similar to that of \cite{Arias13,Maggioni16}.
 A point $p \in M$ being fixed, $T_p M$ is the best local $d$-dimensional linear approximation of $M$ at $p$. Performing a Local Principal Component Analysis (PCA) in a neighborhood of $p$ is likely to recover the main directions spanned by $M$ at $p$, and therefore yield a good approximation of $T_p M$. For $j=1,\ldots,n$ and $h>0$ to be chosen later, define the local covariance matrix at $X_j$ by
	\[
	\hat{\Sigma}_j(h) = \frac{1}{n-1} \sum_{i \neq j} \left ( X_i - \bar{X}_j \right ) \left (X_i - \bar{X}_j\right )^t \mathbbm{1}_{\mathcal{B}(X_j,h)}(X_i),
	\]
 where $\bar{X}_j= \frac{1}{N_j} \sum_{i \neq j} {X_i \mathbbm{1}_{\mathcal{B}(X_j,h)}(X_i)}$ is the barycenter of sample points contained in the ball $\mathcal{B}(X_j,h)$, and  $N_j = |\mathcal{B}(X_j,h)\cap \X_n|$.
 Let us emphasize the fact that the  normalization $1/(n-1)$ in the definition of $\hat{\Sigma}_j$ stands for technical convenience. In fact, any other normalization would yield the same guarantees on tangent spaces since only the principal directions of $\hat{\Sigma}_j$ play a role.
Set $\hat{T}_j(h)$ to be the linear space spanned by the $d$ eigenvectors associated with the $d$ largest eigenvalues of $\hat{\Sigma}_j(h)$. Computing a basis of $\hat{T}_j(h)$ can be performed naively using a singular value decomposition of the full matrix $\hat{\Sigma}_j(h)$, although fast PCA algorithms \cite{Sharma07} may lessen the computational dependence on the ambient dimension. We also denote by $\mathtt{TSE}(.,h)$ the function that maps any vector of points to the vector of their estimated tangent spaces, with
\[
\hat{T}_j(h) = \mathtt{TSE}(\X_n,h)_j.
\] 
\begin{prop}\label{tangent_space_rate_nonoise}
Set $h = \left (c_{d,f_{min},f_{max}} \frac{\log n}{n-1} \right)^{1/d}$ for $c_{d,f_{min},f_{max}}$ large enough. 
Assume that $\sigma/h \leq 1/4$. Then for $n$ large enough, for all $Q \in \mathcal{G}_{D,d,f_{min},f_{max},\rho,\sigma}$,
\[
\max_{1\leq j \leq n} \angle \left( T_{\pi_M(X_j)} {M} , \hat{T}_j(h) \right) \leq C_{d,f_{min},f_{max}} \left( \frac{h}{\rho} + \frac{\sigma}{h} \right),
\]
with probability larger than $1-4\left ( \frac{1}{n} \right )^{\frac{2}{d}}$.
\end{prop}

An important feature given by Proposition \ref{tangent_space_rate_nonoise} is that the statistical error of our tangent space estimation procedure does not depend on the ambient dimension $D$. The intuition behind Proposition \ref{tangent_space_rate_nonoise} is the following: if we assume that the true tangent space $T_{X_j}M$ is spanned by the first $d$ vectors of the canonical basis, we can decompose $\hat{\Sigma}_j$ as
\begin{align*}
\hat{\Sigma}_j(h) =  \left ( \begin{array}{@{}c|c}
      \hat{A}_j(h) & 0 \\
      \hline
      0 & 0
      \end{array}
      \right ) + \hat{R},
\end{align*}
where $\hat{R}$ comes from the curvature of the submanifold along with the additive noise, and is of order $N_j(h)(h^{3}/(\rho(n-1))+h \sigma) \lesssim h^{d+2}(h/\rho + \sigma/h)$, provided that $h$ is roughly smaller than  $(\log(n)/(n-1))^{1/d}$. On the other hand, for a bandwidth $h$ of order $(\log(n)/(n-1))^{1/d}$, $\hat{A}_j(h)$ can be proved (Lemma \ref{Concentration}) to be close to its deterministic counterpart 
\[
A_j(h) = \mathbb{E} \left( \left ( \pi_{T_{X_j} M} (X) - \mathbb{E} \pi_{T_{X_j} M}(X) \right )\left ( \pi_{T_{X_j} M}(X) - \mathbb{E} \pi_{T_{X_j} M}(X) \right )^{t} \mathbbm{1}_{\mathcal{B}(X_j,h)}(X)\right  ),
\]
where $\pi_{T_{X_j} M}$ denotes orthogonal projection onto $T_{X_j}M$ and expectation is taken conditionally on $X_j$. The bandwidth $(\log(n)/(n-1))^{1/d}$ may be thought of as the smallest radius that allows enough sample points in balls to provide an accurate estimation of the covariance matrices. Then, since $f_{min} >0$, Lemma \ref{expectations} shows that the minimum eigenvalue of $A(h)$ is of order $h^{d+2}$. At last, an eigenvalue perturbation result (Proposition \ref{angledeviation}) shows that $\hat{T}_j(h)$ must be close to $T_{X_j}M$ up to $(h^{d+3}/\rho + h^{d+1} \sigma)/(h^{d+2}) \approx h/\rho + \sigma/h$. The complete derivation is provided in Section \ref{proofofpropositiontangentspaceratenonoise}.

Then, it is shown in Lemma \ref{epsilon_rate}, based on the results of \cite{Chazal2013}, that letting $\varepsilon = c_{d,f_{min},f_{max}} (h \vee \rho \sigma/h)$ for $c_{d,f_{min},f_{max}}$ large enough, entails $\X_n$ is $\varepsilon$-dense in $M$ with probability larger than $1 - \left(\frac{1}{n}\right)^{2/d}$. 
Since $\X_n$ may not be sparse at the scale $\varepsilon$, and for the stability reasons described in Section \ref{tangential_delaunay_complex_section}, we sparsify it with the farthest point sampling algorithm (Section \ref{farthest_point_sampling}) with scale parameter $\varepsilon$. 
Let $\Y_n$ denote the output of the algorithm. If $\sigma \leq h/4$, and $c_{d,f_{min},f_{max}}$ is large enough, we have the following.
\begin{cor}\label{recap_nonoise}
With the above notation, for $n$ large enough, with probability at least $1-5\left( \frac{1}{n} \right)^{2/d}$, 
\begin{multicols}{2}
\begin{itemize}
\item $\underset{X_j \in \Y_n}{\max} \: \dd(X_j,M) \leq \frac{\varepsilon^2}{1140 \rho}$,
\item $\underset{X_j \in \Y_n}{\max} \angle ( T_{\pi_M(X_j)} {M},\hat{T}_j(h) ) \leq  \frac{\varepsilon}{2280\rho}$, 
\item $\Y_n$ is $\varepsilon$-sparse,
\item $\underset{x\in M}{\max} \: \dd(x,\Y_n) \leq 2 \varepsilon$.
\end{itemize}
\end{multicols}
\end{cor}
In other words, the previous result shows that $\Y_n$ satisfies the assumptions of Theorem \ref{tdc_stability} with high probability. 
We may then define $\hat{M}_{\mathtt{TDC}}$ to be the Tangential Delaunay Complex computed on $\Y_n$ and the collection of estimated tangent spaces $\mathtt{TSE}(\X_n,h)_{\Y_n}$, that is elements of $\mathtt{TSE}(\X_n,h)$ corresponding to elements of $\Y_n$, where $h$ is the bandwidth defined in Proposition \ref{tangent_space_rate_nonoise}.

\begin{Def} With the above notation, define $\hat{M}_{\mathtt{TDC}} = \mathrm{Del}^{\omega_\ast}\left(\Y_n,\mathtt{TSE}(\X_n,h)_{\Y_n} \right)$.
\end{Def}
Combining Theorem \ref{tdc_stability} and Corollary \ref{recap_nonoise}, it is clear that $\hat{M}_{\mathtt{TDC}}$ satisfies Theorem \ref{main_result_noob_nonoise}.

\subsection{Clutter Noise Case}
\label{clutter_denoise_section}
Let us now focus on the model with outliers $\mathcal{O}_{D,d,f_{min},f_{max},\rho,\beta}$. We address problem of decluttering the sample $\X_n$, that is, to remove outliers. We follow ideas from \cite{Genovese12}. To distinguish whether $X_j$ is an outlier or belongs to $M$, we notice again that points drawn from $M$ approximately lie on a low dimensional structure. On the other hand, the neighborhood points of an outlier drawn far away from $M$ should typically be distributed in an isotropic way. Let $k_1,k_2,h>0$, $x\in \R^D$ and $T\subset \R^D$ a $d$-dimensional linear subspace. The \textit{slab} at $x$ in the direction $T$ is the set $S(x,T,h) = \left\{x\right\} \oplus \mathcal{B}_{T}\left(0,k_1h\right) \oplus  \mathcal{B}_{T^\perp}\!\left(0,k_2h^2\right) \subset \R^D$, where $\oplus$ denotes the Minkovski sum, and $\mathcal{B}_T,\mathcal{B}_{T^\perp}$ are the Euclidean balls in $T$ and $T^\perp$ respectively.
\begin{figure}[H]
\centering
\includegraphics[width=0.75\textwidth]{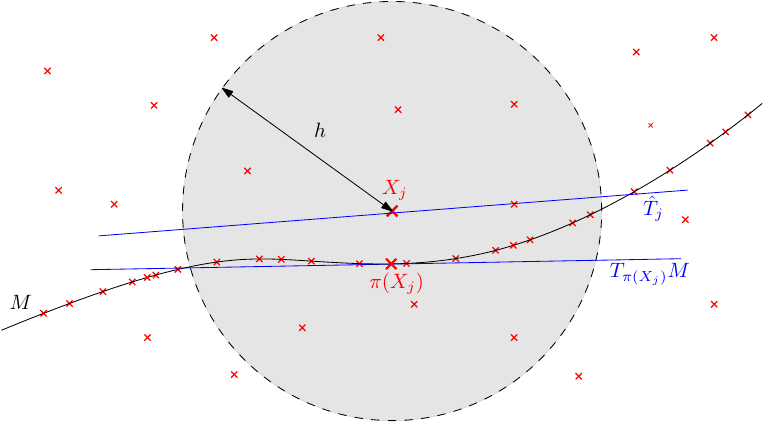}
\caption{Local PCA at an outlier point $X_j \in \X_n$.}
\label{pca_illustration}
\end{figure}
Following notation of Section \ref{statistical_model_section}, for $P \in \mathcal{O}_{D,d,f_{min},f_{max},\rho,\beta}$, let us write $P = \beta Q + (1-\beta)U_{\mathcal{B}_0}$.
For $h$ small enough, by definition of the slabs, $U_{\mathcal{B}_0}\left( S(x,T_{\pi(x)} M,h) \right) \asymp (k_1h)^d (k_2 h^2)^{D-d} \asymp h^{2D-d}$. Furthermore, Figure \ref{slab_illustration} indicates that for $k_1$ and $k_2$ small enough, $Q \left( S(x,T_{\pi(x)} M,h) \right) \asymp Vol\left( S(x,T_{\pi(x)} M,h) \cap M \right) \asymp h^d$ if $\dd(x,M) \leq  h^2$, and $Q \left( S(x,T_{\pi(x)} M,h) \right) = 0$ if $\dd(x,M)> h^2$.
Coming back to $P = \beta Q + (1-\beta)U_{\mathcal{B}_0}$, we roughly get 
\begin{align*}
\begin{array}{@{}lcclrc}
P\left(S(x,T_{\pi(x)} M,h)\right) \asymp & \beta h^d  & + (1-\beta) h^{2D-d} & \asymp h^d & \text{~if~} & \dd(x,M) \leq  h^2,\\
P\left(S(x,T_{\pi(x)} M,h)\right) \asymp & 0 		  & + (1-\beta) h^{2D-d} & \asymp h^{2D-d} &  \text{~if~} & \dd(x,M)> h^2,
 \end{array}
\end{align*}
as $h$ goes to $0$, for $k_1$ and $k_2$ small enough. Since $h^{2D-d} \ll h^{d}$, the measure $P\left(S(x,T,h)\right)$ of the slabs clearly is discriminatory for decluttering, provided that tangent spaces are known.

Based on this intuition, we define the elementary step of our decluttering procedure as the map $\mathtt{SD}_t(.,.,h)$, that sends a vector $P = \left (p_1, \hdots, p_r  \right ) \subset \mathbb{R}^D$ and a corresponding vector of (estimated) tangent spaces $T_\mathcal{P} =( T_{1}, \hdots, T_r )$ onto a subvector of $\mathcal{P}$ according to the rule
        \[
        p_j \in \mathtt{SD}_t(\mathcal{P}, T_\mathcal{P},h) \quad \Leftrightarrow \quad |S(p_j,T_j,h) \cap \mathcal{P} | \geq t(n-1)h^d,
        \]
        where $t$ is a threshold to be fixed. This procedure relies on counting how many sample points lie in the slabs of direction the estimated tangent spaces (see Figure \ref{slab_illustration}).
        \begin{figure}[ht]
\centering
\includegraphics[width=0.9\textwidth]{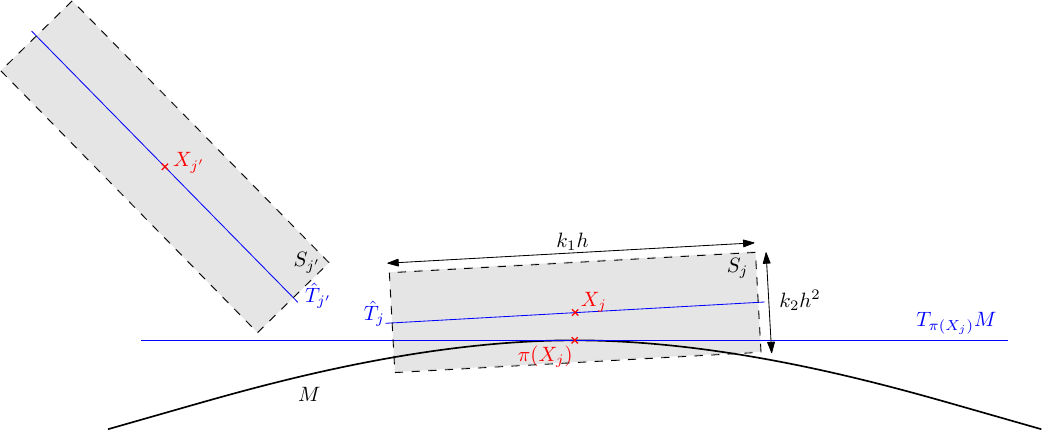}
\caption{The slab $S(X_j,\hat{T}_j,h)$ is centered at $X_j$ and has size $k_1 h$ in the $d$ directions spanned by $\hat{T}_j$, and size $k_2 h^2$ in the $D-d$ directions normal to $\hat{T}_j$.}
\label{slab_illustration}
\end{figure}

Since tangent spaces are unknown, the following result gives some insight on the relation between the accuracy of the tangent space estimation and the decluttering performance that can be reached.

\begin{lem}\label{slabs}
        Let $K>0$ be fixed. There exist constants $k_1(K)$ and $k_2(\rho,K)$ such that for every $h \leq 1$ and $x$ in $\R^D$, $S(x,T,h) \subset \mathcal{B}(x,h/2)$. 
        Moreover, for every $h \leq h_+\wedge 1$ we have 
        \[
        \begin{array}{@{}ccc}        
        h / \sqrt{2} \geq \dd(x,M) \geq h^2 / \rho \quad \mbox{and} \quad \angle \left (T_{\pi_M(x)} M, T \right ) \leq K h / \rho & \Rightarrow & S(x,T,h) \subset S'(x,T_{\pi_M(x)}M,h), 
        \end{array}        
        \]
		where $S'(x,T_{\pi_M(x)}M,h)$ is a larger slab with parameters $k'_1(\rho,K)$ and $k'_2(\rho,K)$,  which are such that $S'(x,T_{\pi_M(x)}M,h) \cap M = \emptyset$.
		 In addition, there exists $k_3(\rho,K)$ such that for all $x$ and $y$ are in $M$,
        \[
        \angle \left (T_x M, T \right ) \leq K h / \rho \quad \mbox{and} \quad \norm{x-y} \leq k_3 h \quad \Rightarrow \quad y \in S(x,T,h).
        \]
         Possible values for $k_1$ and $k_2$ are, respectively, $\frac{1}{ 16(K \vee1)}$ and $\frac{1}{16(\rho \vee K \vee 1)}$, and $k_3$ can be taken as $ k_1 \wedge \frac{\rho k_2}{1 + 2K}$.
        \end{lem}   
        The proof of Lemma  \ref{slabs}, mentioned in \cite{Genovese12}, follows from elementary geometry, combined with the definition of the reach and Proposition \ref{normalpart}.
        
        Roughly, Lemma \ref{slabs} states that the decluttering performance is of order the square of the tangent space precision, hence will be closely related to the performance of the tangent space estimation procedure $\mathtt{TSE}$. Unfortunately, a direct application of $\mathtt{TSE}$ to the corrupted sample $\X_n$ leads to slightly worse precision bounds, in terms of angle deviation. Typically, the angle deviation would be of order $n^{-1/(d+1)}$. However, this precision is enough to remove outliers points which are at distance at least $n^{-2/(d+1)}$ from $M$. Then running our $\mathtt{TSE}$ on this refined sample $\mathtt{SD}_t(\X_n,\mathtt{TSE}(\X_n),n^{-1/(d+1)})$ leads to better angle deviation rates, hence better decluttering performance, and so on.

Let us introduce an iterative decluttering procedure in a more formal way. We choose the initial bandwidth $h_0 = \left (  c_{d,f_{min},f_{max},\rho} \frac{\log n}{\beta (n-1)} \right )^{\gamma_0}$, with $\gamma_0 = 1/(d+1)$, \label{def_h0} and define the first set $\X^{(-1)} = \X_n$ as the whole sample. We then proceed recursively, setting $h_{k+1} = \left (  c_{d,f_{min},f_{max},\rho} \frac{\log n}{\beta (n-1)} \right )^{\gamma_{k+1}}$, with $\gamma_{k+1}$ satisfying $\gamma_{k+1} = (2 \gamma_k +1)/(d+2)$. This recursion formula is driven by the optimization of a trade-off between imprecision terms in tangent space estimation, as may be seen from \eqref{biasvariance}. An elementary calculation shows that
\[
\gamma_k = \frac{1}{d} - \frac{1}{d(d+1)}\left(\frac{2}{d+2}\right )^k.
\]
With this updated bandwidth we define
         \[
         \X^{(k+1)} = \mathtt{SD}_t(\X^{(k)}, \mathtt{TSE}(\X^{(k)},h_{k+1}),h_{k+1}).
         \]
     In other words, at step $k +1$ we use a smaller bandwidth $h_{k+1}$ in the tangent space estimation procedure $\mathtt{TSE}$. Then we use this better estimation of tangent spaces to run the elementary decluttering step $\mathtt{SD}$. The performance of this procedure is guaranteed by the following proposition. With a slight abuse of notation, if $X_j$ is in $\X^{(k)}$, $\mathtt{TSE}(\X^{(k)},h)_j$ will denote the corresponding tangent space of $\mathtt{TSE}(\X^{(k)},h)$.

\begin{prop}\label{noisefiltering}
In the clutter noise model, for $t$, $c_{d,f_{min},f_{max},\rho}$ and $n$ large enough, $k_1$ and $k_2$ small enough, the following properties hold with probability larger than $1-7\left( \frac{1}{n} \right )^{2/d}$ for all $k\geq 0$.
             \medskip
             
             \textbf{Initialization}: 
             \begin{itemize}
             \item For all $X_j$ $\in$ $\X^{(-1)}$ such that $\dd(X_j,M) \leq h_0/\sqrt{2}$, $$\angle \bigl( \mathtt{TSE}(\X^{(-1)},h_0)_j, T_{\pi(X_j)} {M} \bigr) \leq C_{d,f_{min},f_{max}} h_0/\rho.$$
                          \item  For every $X_j$ $\in$ $M \cap \X^{(-1)}$, $X_j$ $\in$ $\X^{(0)}$.
             \item For every $X_j$ $\in$ $\X^{(-1)}$, if $\dd(X_j,M) > h_0^2/\rho$, then $X_j$ $\notin$ $\X^{(0)}$.
            
             \end{itemize}
             
             \medskip
             
             \textbf{Iterations}:
             \begin{itemize}
             \item For all $X_j$ $\in$ $\X^{(k)}$ such that $\dd(X_j,M) \leq h_{k+1}/\sqrt{2}$, $$\angle \bigl( \mathtt{TSE}(\X^{(k)},h_{k+1})_j, T_{\pi(X_j)} {M} \bigr) \leq C_{d,f_{min},f_{max}} h_{k+1}/\rho.$$
                          \item  For every $X_j$ $\in$ $M \cap \X^{(k)}$, $X_j$ $\in$ $\X^{(k+1)}$.
             \item For every $X_j$ $\in$ $\X^{(k)}$, if $\dd(X_j,M) > h_{k+1}^2/\rho$, then $X_j$ $\notin$ $\X^{(k+1)}$.
            
             \end{itemize}
             \end{prop}  
This result is threefold. Not only can we distinguish data and outliers within a decreasing sequence of offsets of radii $h_k^2/\rho$ around $M$,           but we can also ensure that no point of $M$ is removed during the process with high probability. Moreover, it also provides a convergence rate for the estimated tangent spaces $\mathtt{TSE}(\X_k,h_{k+1})$.   
                    
Now fix a precision level $\delta$. If $k$ is larger than $(\log(1/\delta) - \log(d(d+1)) /(\log(d+2) - \log(2))$, then $1/d > \gamma_k \geq 1/d - \delta$. Let us define $k_\delta$ as the smallest integer satisfying $\gamma_{k} \geq 1/d -\delta$, and denote by $\Y_n^\delta$ the output of the farthest point sampling algorithm applied to $\X^{(k_\delta)}$ with parameter $\varepsilon = c_{d,f_{min}f_{max}} h_{k_\delta}$, for $c_{d,f_{min}f_{max}}$ large enough. Define also $\hat{T}^\delta$ as the restriction of $\mathtt{TSE}(\X^{(k_\delta)},h_{k_\delta})$ to the elements of $\Y_n^{\delta}$.

According to Proposition \ref{noisefiltering}, the decluttering procedure removes no data point on $M$ with high probability. In other words, $\X^{(k_\delta)} \cap M = \X_n \cap M$, and as a consequence, $\underset{x\in M}{\max} \: \dd(x,\X^{(k_\delta)}) \leq c_{d,f_{min}} \left( \frac{\log n}{\beta n}\right)^{1/d} \ll h_{k_\delta}$ with high probability (see Lemma \ref{epsilon_rate}). As a consequence, we obtain the following.

\begin{cor}\label{recap_noise}
With the above notation, for $n$ large enough, with probability larger than $1- 8 \left( \frac{1} n \right)^{2/d}$,
\begin{multicols}{2}
\begin{itemize}
\item $\underset{X_j \in \Y_n^\delta}{\max} \dd(X_j,M) \leq \frac{\varepsilon^2}{1140 \rho}$,
\item $\underset{X_j \in \Y_n^\delta}{\max} \angle ( T_{\pi_M(X_j)} {M} , \hat{T}^\delta_j ) \leq \frac{\varepsilon}{2280\rho}$, 
\item $\Y_n^\delta$ is $\varepsilon$-sparse,
\item $\underset{x\in M}{\max} \: \dd(x,\Y_n^\delta) \leq 2 \varepsilon$.
\end{itemize}
\end{multicols}
\end{cor}
We are now able to define the estimator $\hat{M}_{\mathtt{TDC \delta}}$.
\begin{Def}
\label{final_estimator_definition}
With the above notation, define
$
\hat{M}_{\mathtt{TDC \delta}} = \mathrm{Del}^{\omega_\ast}\left(\Y^\delta_n,\hat{T}^\delta \right).
$
\end{Def}
Combining Theorem \ref{tdc_stability} and Corollary \ref{recap_noise}, it is clear that $\hat{M}_{\mathtt{TDC \delta}}$ satisfies Theorem \ref{main_result_noob_noise}.

Finally, we turn to the asymptotic estimator $\hat{M}_{\mathtt{TDC +}}$. Set $h_{\infty} = \left (  c_{d,f_{min},f_{max},\rho} \frac{\log n}{\beta (n-1)} \right )^{1/d}$, and let $\hat{k}$ denote the smallest integer such that $\min \{ \dd(X_j,M) | \dd(X_j,M) > h_{\infty}^2/\rho \} > h_{\hat{k}}^2/\rho$. Since $\X_n$ is a (random) finite set, we can always find such a random integer $\hat{k}$ that provides a sufficient number of iterations to obtain the asymptotic decluttering rate. For this random iteration $\hat{k}$, we can state the following result.

\begin{prop}\label{infinite_denoising_estimator}
Under the assumptions of Corollary \ref{recap_noise}, for every $X_j$ $\in$ $X^{(\hat{k} +1)}$, we have
\[
\angle ( \mathtt{TSE}(\X^{(\hat{k}+1)},h_{\infty})_j, T_{\pi(X_j)} {M} ) \leq C_{d,f_{min},f_{max}} h_{\infty}/\rho.
\] 
\end{prop}
As before, taking $\Y_n^{+}$ as the result of the farthest point sampling algorithm based on $\X^{(\hat{k}+1)}$, and $T^+$ the vector of tangent spaces  $\mathtt{TSE}(\X^{(\hat{k}+1)},h_{\infty})_j$ such that $\X^{(\hat{k}+1)}_j $ $\in$ $\Y_n^+$, we can construct our last estimator.

\begin{Def}
\label{theoretical_estimator_definition}
With the above notation, define
$
\hat{M}_{\mathtt{TDC +}} = \mathrm{Del}^{\omega_\ast}\left(\Y^+_n,T^{+} \right).
$
\end{Def} 
 In turn, Proposition \ref{infinite_denoising_estimator} implies that $\hat{M}_{\mathtt{TDC +}}$ satisfies Theorem \ref{theoretical_rate}. 

\section{Conclusion}
\label{conclusion}
In this work, we gave results on explicit manifold reconstruction with simplicial complexes. We built estimators $\hat{M}_{\mathtt{TDC}}$, $\hat{M}_{\mathtt{TDC \delta}}$ and $\hat{M}_{\mathtt{TDC+}}$ in two statistical models. We proved minimax rates of convergence for the Hausdorff distance and consistency results for ambient isotopic reconstruction. Since $\hat{M}_{\mathtt{TDC}}$ is minimax optimal in the additive noise model for $\sigma$ small, and uses the Tangential Delaunay Complex of \cite{Boissonnat14}, the latter is proved to be optimal. Moreover, rates of \cite{Genovese12} are proved to be achievable with simplicial complexes that are computable using existing algorithms. To prove the stability of the Tangential Delaunay Complex, a generic interpolation result was derived. In the process, a tangent space estimation procedure and a decluttering method both based on local PCA were studied.

In the model with outliers, the proposed reconstruction method achieves a rate of convergence that can be as close as desired to the minimax rate of convergence, depending on the number of iterations of the decluttering procedure. Though this procedure seems to be well adapted to our reconstuction scheme --- which is based on tangent spaces estimation --- we believe that it could be of interest in the context of other applications. Also, further investigation may be carried out to compare this decluttering procedure to existing ones \cite{Buchet15,Donoho95}. 

As briefly mentioned below Theorem \ref{main_result_noob_nonoise}, our approach is likely to be suboptimal in cases where noise level $\sigma$ is large. In such cases, with additional structure on the  noise such as \textit{centered} and \textit{independent from the source}, other statistical procedures such as deconvolution \cite{Genovese12} could be adapted to provide vertices to the Tangential Delaunay Complex. Tangential properties of deconvolution are still to be studied.

The effective construction of $\hat{M}_{TDC \delta}$ can be performed using existing algorithms. Namely, Tangential Delaunay Complex, farthest point sampling, local PCA and point-to-linear subspace distance computation for slab counting. A crude upper bound on the time complexity of a naive step-by-step implementation is
\[
O\left( nD \left [ 2^{O(d^2)} + \log (1/\delta) D(D+n) \right ] \right),
\]
since the precision $\delta$ requires no more than $\log \left(1/\delta\right)$ iterations of the decluttering procedure.
It is likely that better complexity bounds may be obtained using more refined algorithms, such as fast PCA \cite{Sharma07}, that lessens the dependence on the ambient dimension $D$.
An interesting development would be to investigate a possible precision/complexity tradeoff, as done in \cite{Arias14} for community detection in graphs for instance.

Even though Theorem \ref{tangential_perturbation_thm} is applied to submanifold estimation, we believe it may be applied in various settings. Beyond its statement, the way that it is used is quite general. When intermediate objects (here, tangent spaces) are used in a procedure, this kind of proxy method can provide extensions of existing results to the case where these objects are only approximated. 

As local PCA is performed throughout the paper, the knowledge of the bandwidth $h$ is needed for actual implementation. In practice its choice is a difficult question and adaptive selection of $h$ remains to be considered.

In the process, we derived rates of convergence for tangent space estimation. The optimality of the method will be the object of a future paper.

\paragraph*{acknowledgements}
We would like to thank Jean-Daniel Boissonnat, Fr\'ed\'eric Chazal, Pascal Massart, and Steve Oudot for their insight and the interest they brought to this work.
We are also grateful to the reviewers whose comments helped enhancing substantially this paper.

This work was supported by ANR project TopData ANR-13-BS01-0008 and by the Advanced Grant of the European Research Council GUDHI (Geometric Understanding in Higher
Dimensions). E. Aamari was supported by the Conseil r\'egional d'\^{I}le-de-France under a doctoral allowance of its program R\'eseau de Recherche Doctoral en Math\'ematiques de l'\^{I}le-de-France (RDM-IdF).

\newpage
\bibliography{Biblio_Tangential_Stability_Final}


\appendix

\section{Interpolation Theorem}

This section is devoted to prove the interpolation results of Section \ref{interpolation_theorem_section}. For sake of completeness, let us state a stability result for the reach with respect to $\mathcal{C}^2$-diffeomorphisms.

\begin{lem}[Theorem 4.19 in \cite{federer1959}]
\label{reach_stability}
Let $A \subset \mathbb{R}^D$ with $\mathrm{reach}(A) \geq \rho > 0$ and $ \Phi  : \mathbb{R}^D \longrightarrow \mathbb{R}^D$ is a $\mathcal{C}^1$-diffeomorphism such that $ \Phi $,$ \Phi ^{-1}$, and $\dd  \Phi $ are Lipschitz with Lipschitz constants $K$,$N$ and $R$ respectively, then
	\[ \mathrm{reach}(\Phi(A)) \geq \dfrac{1}{(K\rho^{-1} + R)N^2}. \]
\end{lem}
Writing $\phi_\ell(\cdot) = \phi(\cdot/\ell)$, we recall that ${\psi_j(a)} = { (R_j-I_D)(a - \pi(p_j)) + (p_j - \pi(p_j)) }$ and
\begin{align}
\Phi(a) &= a + \sum_{j=1}^q \phi_\ell \left(a-\pi(p_j) \right) \psi_j(a).
\label{bump_defi}
\end{align}
Let us denote $b_1 = \sup_x \norm{\dd_x \phi}$, $b_2 = \sup_x \norm{\dd^2_x \phi}_\mathrm{op}$, and write $C_1 = 1+b_1$, $C_2 = b_2+2b_1$. Straightforward computation yields $C_1 \leq 7/2$ and $C_2 \leq 28$.

\begin{proof}[\proofof Lemma \ref{bump_lem}]
	First notice that the sum appearing in (\ref{bump_defi}) consists of at most one term. Indeed, since $\phi \equiv 0$ outside $\mathcal{B}(0,1)$, if $\phi_\ell \left(a-\pi(p_j) \right) \neq 0$ for some $j \in \{1,\ldots,q\}$,  then $\norm{a - \pi(p_j)} \leq \ell$. Consequently, for all $i \neq j$,
	\begin{align*}
	\norm{a - \pi(p_i)} & \geq \norm{p_j - p_i} - \norm{p_j - \pi(p_j)} - \norm{\pi(p_j) - a} - \norm{\pi(p_i) - p_i}   \\
		& \geq \delta - \eta - \ell - \eta  \\
		& \geq \delta - 2\ell \geq \ell,
	\end{align*}
	where we used that $ 6 \eta \leq \ell \leq \delta / 3$. Therefore, $\phi_\ell \left(a-\pi(p_i) \right) = 0$ for all $i\neq j$. In other words, if a $p_j$ actually appears in $\Phi(a)$ then the others do not.
	
	\textit{Global diffeomorphism:} As the sum in (\ref{bump_defi}) is at most composed of one term, chain rule yields
		\begin{align*}
			\norm{\dd_a \Phi - {{I_D}}}_\mathrm{op} 
			&= \underset{ 1 \leq j \leq q }{\max} \norm{\dd_a \left[ \phi_\ell \left(a-\pi(p_j) \right) \psi_j(a) \right]}_\mathrm{op} \\
			& = \underset{ 1 \leq j \leq q }{\max} \norm{\psi_j(a) \left. \frac{\dd_b \phi}{\ell}\right|_{b = \frac{a-\pi(p_j)}{\ell}} + \phi_\ell \left(a-\pi(p_j) \right)(R_j-I_D)}_\mathrm{op} \\
			& \leq \bigl(b_1+1\bigr)\theta + b_1\frac{\eta}{\ell} < 1,
		\end{align*}	
where the last line follows from $b_1 \leq 5/2$, $6 \eta \leq \ell$ and $\theta \leq \pi/64$.
Therefore, $\dd_a  \Phi $ is invertible for all $a\in \R^D$, and $\left(\dd_a  \Phi \right)^{-1} = \sum_{i=0}^\infty \left(I_D- \dd_a  \Phi \right)^i$. $ \Phi $ is a local diffeomorphism  according to the local inverse function theorem. Moreover, $\norm{\Phi(a)} \rightarrow \infty$ as $\norm{a} \rightarrow \infty$, so that $ \Phi $ is a global $\mathcal{C^\infty}$-diffeomorphism by Hadamard-Cacciopoli theorem \cite{DeMarco1994}. 
		
	\textit{Differentials estimates:}	
	\textit{(i) First order:} From the estimates above,	
		\begin{align*}
			\norm{\dd_a  \Phi }_\mathrm{op}  \leq \norm{I_D}_{op} + \norm{\dd_a  \Phi  - I_D}_{op} \leq 1+ \bigl(b_1+1\bigr)\theta + b_1\frac{\eta}{\ell}.
		\end{align*}
		
	\textit{(ii) Inverse:}  Write for all $a \in \R^D$,
	\begin{align*}
			\norm{\dd_{\Phi(a)}  \Phi ^{-1}}_\mathrm{op} &= \norm{(\dd_a  \Phi )^{-1}}_\mathrm{op}
							 =\norm{ \sum_{i=0}^\infty \left(I_D- \dd_a  \Phi \right)^i}_\mathrm{op} \\ 
							 &\leq \frac{1}{1- \norm{{{I_D}} - \dd_a  \Phi }_\mathrm{op} } \leq \frac{1}{1- \bigl(b_1+1\bigr)\theta - b_1\frac{\eta}{\ell} },
	\end{align*}
	
	where the first inequality holds since $\norm{\dd_a \Phi - {{I_D}}}_\mathrm{op} < 1$, and $\norm{\cdot}_\mathrm{op}$ is sub-multiplicative.
		
		\textit{(iii) Second order:} Again, since the sum (\ref{bump_defi}) includes at most one term,
		\begin{align*}
			\norm{\dd^2_a \Phi}_\mathrm{op} 
			&= \underset{ 1 \leq j \leq q }{\max} \norm{\dd^2_a \left[ \phi_\ell \left(a-\pi(p_j) \right) \psi_j(a) \right]}_\mathrm{op} \\
			& \leq \underset{ 1 \leq j \leq q }{\max}  \left\{
				\frac{\norm{\dd^2 \phi}_\mathrm{op}}{\ell^2} \norm{\psi_j(a)} + 2 \frac{\norm{\dd \phi}_\mathrm{op}}{\ell} \norm{R_j-I_D}_\mathrm{op} \right\}
				 \\
			& \leq  b_2 \frac{\eta}{\ell^2} + \left(b_2 + 2b_1 \right) \frac{\theta}{\ell}.
		\end{align*}
\end{proof}

\begin{proof}[\proofof Theorem \ref{tangential_perturbation_thm}]
Set $\ell = \delta/3$ and $M' =  \Phi (M)$.

\begin{itemize}
\item \textit{Interpolation:} For all $j$,  $p_j =  \Phi (\pi(p_j)) \in M'$ by construction since $\phi_\ell(0)=1$.

\item \textit{Tangent spaces:} Since $\left.\dd_x \phi_l\right|_{x=0} = 0$, for all $j \in \bigl\{1,\ldots,q \bigr\}$, $\left.\dd_a  \Phi \right|_{a=\pi(p_j)} = R_j$. Thus, 
	\begin{align*}
	T_{p_j} {M}' &= T_{ \Phi (\pi(p_j))}  \Phi (M)  \\
		&= \left.\dd_a  \Phi \right|_{a=\pi(p_j)} \left(T_{\pi(p_j)} M \right)\\
		&=  R_j \left(T_{\pi(p_j)} M \right) = T_j,
	\end{align*}
by definition of $R_j$.

\item \textit{Proximity to $M$:} The bound on $d_H(M,M') = d_H\bigl(M, \Phi(M) \bigr)$ follows from the correspondence
 \begin{align*}
\norm{\Phi(a) - a} &\leq \underset{a \in \R^D}{\sup} \underset{1\leq j\leq q}{\max} \phi_\ell \left(a-\pi(p_j) \right) \norm{\psi_j(a)} \\
				&\leq  \ell \theta + \eta \leq \delta \theta + \eta.
\end{align*}

\item \textit{Isotopy:} Consider the continuous family of maps
	\[  \Phi _{(t)}(a) = a + t \left( \sum_{j = 1}^q \phi_\ell \left(a-\pi(p_j) \right) {\psi_j(a)} \right),
	\]
	for $0 \leq t \leq 1$. Since $ \Phi _{(t)} - {{I_D}} = t \bigl( \Phi  - {{I_D}}\bigr)$, the arguments above show that $ \Phi _{(t)}$ is a global diffeomorphism of $\R^D$ for all $t \in [0,1]$. Moreover $ \Phi _{(0)} = {{I_D}}$, and $ \Phi _{(1)} =  \Phi $. Thus, $M =  \Phi _{(0)}(M)$ and $M' =  \Phi _{(1)}(M)$ are ambient isotopic.
	
\item \textit{Reach lower bound:} The differentials estimates of order $1$ and $2$ of $ \Phi $ translate into estimates on Lipschitz constants of $ \Phi $,$ \Phi ^{-1}$ and $\dd  \Phi $. Applying  Lemma \ref{reach_stability} leads to
\begin{align*}
	\mathrm{reach}\left(M'\right) 
		& \geq \dfrac{\left( 1-C_1 \left(\frac{\eta}{\ell} + \theta \right) \right)^2}{\dfrac{1+ C_1 \left(\frac{\eta}{\ell} + \theta \right)}{\rho}+ C_2 \left(\frac{\eta}{\ell^2} +\frac{\theta}{\ell}\right)} 
		 = \rho \cdot  \dfrac{\left( 1-C_1 \left(\frac{\eta}{\ell} + \theta \right) \right)^2}{1+ C_1 \left(\frac{\eta}{\ell} + \theta \right)+ C_2 \left(\frac{\eta}{\ell^2} +\frac{\theta}{\ell}\right) \rho}.
\end{align*}
Now, replace $\ell$ by its value $\delta/3$, and write $c_1 = 3 C_1 \leq 21/2 \leq 11$ and $c_2 = 3^2 C_2 \leq 252$.
We derive 
\begin{align*}
\mathrm{reach}\left(M'\right) 
&\geq
\left(1-2c_1 \left(\frac{\eta}{\delta} + \theta \right) \right) 
\left( 
	1
	- c_1 \left(\frac{\eta}{\delta} + \theta \right)
	- c_2 \left(\frac{\eta}{\delta^2} + \frac{\theta}{\delta} \right) \rho
\right)
\rho
\\
&\geq
\left( 
	1
	- 3 c_1 \left(\frac{\eta}{\delta} + \theta \right)
	- c_2 \left(\frac{\eta}{\delta^2} + \frac{\theta}{\delta} \right) \rho
\right)
\rho
\\
&\geq
\left( 
	1
	- \left(3 c_1 +c_2\right) \left(\frac{\eta}{\delta^2} + \frac{\theta}{\delta} \right) \rho
\right)
\rho
,
\end{align*}
where for the last line we used that $\delta / \rho \leq 1$. The desired lower bound follows taking $c_0 = 3c_1+ c_2 \leq 285$.
\end{itemize}

\end{proof}

\section{Some Geometric Properties under Reach Regularity Condition}
 
\subsection{Reach and Projection on the Submanifold}
In this section we state intermediate results that connect the reach condition to orthogonal projections onto the tangent spaces. They are based on the following fundamental result.

\begin{prop}[Theorem 4.18 in \cite{federer1959}]\label{normalpart}
For all $x$ and $y$ in $M$, 
\[
\|(y-x)_{\perp}\| \leq \frac{\|y-x\|^2}{2 \rho},
\]
where $(y-x)_{\perp}$ denotes the projection of $y-x$ onto $T_x M^{\perp}$.
\end{prop}

        From Proposition \ref{normalpart} we may deduce the following property about trace of Euclidean balls on $M$.
      
\begin{prop}\label{ballprojection}
     Let $x \in \mathbb{R}^D$  be such that $\dd(x,M) = \Delta \leq h \leq \frac{\rho}{8}$, and let $y$ denote $\pi(x)$. Then,
     \begin{align*}
     \mathcal{B}\left (y, r_h^- \right )\cap M \subset \mathcal{B}(x,h) \cap M \subset \mathcal{B} \left ( y, r_h^+ \right ) \cap M,
     \end{align*}
     where $r_h^2 + \Delta^2 = h^2$, $(r_h^-)^2 = \left (1-\frac{ \Delta}{\rho}\right )r_h^2$, and $(r_h^+)^2 = \left (1+\frac{ 2 \Delta}{\rho}\right )r_h^2$.
\end{prop}

\begin{proof}[\proofof Proposition \ref{ballprojection}]
Let $z$ be in $M \cap \mathcal{B}(x,h)$, and denote by $\delta$ the quantity $\|z-y\|$. We may write
\begin{align}\label{eqballproj}
       \|z-x\|^2 = \delta^2 + \Delta^2 +2 \left\langle z-y,y-x \right\rangle,
\end{align}
hence $\delta^2 \leq h^2 - \Delta^2 - 2 \left\langle z-y,y-x \right\rangle$. Denote, for $u$ in $\mathbb{R}^D$, by $u_\perp$ its projection onto $T_y M^{\perp}$. Since $\left\langle z-y,y-x \right\rangle = \left\langle (z-y)_{\perp},y-x \right\rangle$, Proposition \ref{normalpart} ensures that
\[
\delta^2\left ( 1 - \frac{\Delta}{\rho} \right ) \leq r_h^2.
\] 
 Since $\Delta \leq h \leq \rho/8$, it comes $\delta^2 \leq (1+2 \frac{\Delta}{\rho})r_h^2$. On the other hand, \eqref{eqballproj} and Proposition \ref{normalpart} also yield
 \[
 \|z-x\|^2 \leq \delta^2 \left ( 1 + \frac{\Delta}{\rho} \right ) + \Delta^2.
 \]
 Hence, if $\delta^2 \leq \left (1- \frac{\Delta}{\rho} \right )r_h^2$, we have
 \[
 \|z-x\|^2 \leq r_h^2 + \Delta^2 = h^2.
 \]

\end{proof}

Also, the following consequence of Proposition \ref{normalpart} will be of particular use in the decluttering procedure.

\begin{prop}\label{normalnoise}
Let $h$ and $h_k$ be bandwidths satisfying $h_k^2/\rho \leq h \leq h_k$. Let $x$ be such that $d(x,M) \leq h/\sqrt{2}$ and $\pi_{M} (x) = 0$, and let $z$ be such that $\|z-x\| \leq h$ and $d(z,M) \leq h_k^2/\rho $. Then
\[
\| z_{\perp} \| \leq \frac{6 h_k^2}{\rho},
\]
where $z_{\perp}$ denotes the projection of $z$ onto $T_0 M^{\perp}$.
\end{prop}
\begin{proof}[\proofof Proposition \ref{normalnoise}]
         Let $y$ denote $\pi_M (z)$. A triangle inequality yields $\|y\| \leq \| y-z \| + \|z-x\| + \|x\| \leq h_k^{2}/\rho + (1 + 1/\sqrt{2}) h \leq 3 h_k$.         
Proposition \ref{normalpart} ensures that $\| y_{\perp} \| \leq  \| y \|^2/(2\rho) \leq (9 h_k^2)/(2\rho)$. Since $\| z_{\perp} \| \leq \|y_{\perp}\| + h_k^2/\rho$, we have $\| z_{\perp} \| \leq 6 h_k^2/\rho$ .      
\end{proof}

At last, let us prove Lemma \ref{slabs}, that gives properties of intersections of ambient slabs with $M$.

\begin{proof}{(\proofof Lemma \ref{slabs})}
Set $k_1=\frac{1}{16(K \vee 1)}$, $k_2 = \frac{1}{16(K \vee \rho \vee 1)}$, and $k_3 = k_1 \wedge \frac{\rho k_2}{1+2K} \wedge 1$. For all $h> 0$, and $z \in S(x,T,h)$, triangle inequality yields $\|z-x\| \leq \|\pi_T(z-x)\| + \|\pi_{T^{\perp}}(z-x)\| \leq (k_1 + k_2h)h$. Since $h \leq 1$ and $k_1 + k_2 \leq 1/2$, we get $z \in \mathcal{B}(x,h/2)$.
  
   Now, suppose that $ h/\sqrt{2} \geq \dd(x,M) \geq h^2 / \rho$ and $\angle \left (T_{\pi(x)} M, T \right ) \leq K h / \rho$. For short we write $T_0 = T_{\pi(x)} M$. Let $z \in S(x,T,h)$, since $h \leq 1$,  it comes 
   \begin{align*}
   \| \pi_{T_0} (z-x) \| \leq \| z-x \| \leq (k_1 + k_2)h = k'_1 h,
   \end{align*}
   with $k'_1 = k_1 + k_2$. On the other hand
   \begin{align*}
   \| \pi_{T_0^{\perp}} (z-x) \| & \leq \| \pi_{T_0^{\perp}} \pi_{T} (z-x) \| + \| \pi_{T_0^{\perp}} \pi_{T^{\perp}} (z-x) \| \leq (Kh/\rho)(k_1 h) + k_2 h^2 = k'_2 h^2,
   \end{align*}
   with $k'_2 = k_1 K/\rho + k_2$. Hence $S(x,T,h) \subset S'(x,T_0,h)$, for the constants $k'_1$ and $k'_2$ defined above. It remains to prove that $S'(x,T_0,h) \cap M = \emptyset$. To see this, let $z \in S'(x,T_0,h)$, and $y = \pi(x)$. Since $k'_1 + k'_2 \leq 1/4$, we have $\| y-z \| \leq \| y-x \| + \| x-z \| \leq h(1/\sqrt{2} + 1/4)$.
   For the normal part, we may write 
   \begin{align*}
   \| \pi_{T_0^{\perp}}(z-y) \| \geq \| \pi_{T_0^{\perp}}(y-x)\| - \| \pi_{T_0^{\perp}}(x-z) \| \geq h^2(1/\rho - k'_2).
   \end{align*}
   Since $k'_2 \leq 1/(8\rho)$, we have $\| \pi_{T_0^{\perp}}(z-y) \| > \|y-z\|^2 /(2\rho)$, hence Proposition \ref{normalpart} ensures that $z \notin M$.
   
   At last, suppose that $x \in M$ and $y \in \mathcal{B}(x,k_3h)\cap M$. Since $k_3 \leq k_1$, we have $\| \pi_{T}(y-x)\| \leq k_1h$. Next, we may write
   \begin{align*}
   \| \pi_{T^{\perp}}(y-x) \| \leq \| \pi_{T^{\perp}} \pi_{T_0}(y-x) \| + \| \pi_{T^{\perp}}\pi_{T_0^{\perp}}(y-x) \|.
\end{align*}    
Since $y \in M$, Proposition \ref{normalpart} entails $\|\pi_{T_0^{\perp}}(y-x) \| \leq \|y-x\|^2/(2\rho) \leq k_3^2 h^2/(2\rho)$. It comes 
\[
\| \pi_{T^{\perp}}(y-x) \| \leq \frac{h^2}{\rho} \left ( k_3 K + \frac{k_3^2}{2} \right ) \leq k_2 h^2.
\]
Hence $y \in S(x,T,h)$.
\end{proof}

\subsection{Reach and Exponential Map}

In this section we state results that connect Euclidean and geodesic quantities under reach regularity condition. We start with a result linking reach and principal curvatures.

\begin{prop}[Proposition 6.1 in \cite{Niyogi08}]
For all $x\in M$, writing $II_x$ for the second fundamental form of ${M}$ at $x$, for all unitary $w\in T_x {M}$, we have $\norm{II_x(w,w)} \leq 1/\rho$.
\end{prop}
For all $x \in {M}$ and $v\in T_x {M}$, let us denote by $\exp_x(v)$ the exponential map at $x$ of direction $v$. According to the following proposition, this exponential map turns out to be a diffeomorphism on balls of radius at most $\pi \rho$.

\begin{prop}[Corollary 1.4 in \cite{Alexander06}]\label{injectivityradius}
The injectivity radius of $M$ is at least $\pi \rho$.
\end{prop}
 Denoting by $\dd_M(\cdot,\cdot)$  the geodesic distance on $M$, we are in position to connect geodesic and Euclidean distance.
In what follows, we fix the constant $\alpha = 1 + \frac{1}{4 \sqrt{2}}$.

\begin{prop}\label{geodesics}
For all $x,y \in {M}$ such that $\norm{x-y} \leq \rho/4$,
\[ \norm{x-y} \leq \dd_{M}(x,y) \leq \alpha \norm{x-y}.\]
Moreover, writing $y = \exp_x(r v)$ for $v\in T_x {M}$ with $\norm{v}=1$ {\color{black}{and $r \leq \rho/4$}},
\[ y = x +rv + R(r,v)\]
with $\norm{R(r,v)} \leq \frac{r^2}{2\rho}$.
\end{prop}

\begin{proof}[\proofof Proposition \ref{geodesics}]
The first statement is a direct consequence of Proposition 6.3 in \cite{Niyogi08}. 
  Let us define $u(t) =\exp_x(tv) - \exp_x(0) -tv$ and $w(t) = \exp_x(tv)$ for all $0\leq t \leq r$. It is clear that $u(0) = 0$ and $u'(0) = 0$. 
Moreover, $\norm{u''(t)} = \norm{II_{w(t)}\left(w'(t),w'(t)\right)} \leq 1/\rho$. Therefore, a Taylor expansion at order two gives $\norm{R(r,v)} = u(r) \leq  r^2/(2\rho)$. Applying the first statement of the proposition gives $r\leq \alpha \norm{x-y}$. 
\end{proof}
      The next proposition gives bounds on the volume form expressed in polar coordinates in a neighborhood of points of ${M}$.
      
      \begin{prop}\label{volumeform}
      Let $x \in {M}$ be fixed. Denote by $J(r,v)$ the Jacobian of the volume form expressed in polar coordinates around $x$, for $r \leq \frac{\rho}{4}$ and $v$ a unit vector in $T_x {M}$. In other words, if $y = \exp_x(rv)$, $d_yV =J(r,v) dr dv.$
      Then  \[c_d r^{d-1} \leq J(r,v) \leq C_d r^{d-1},\] where $c_d = 2^{-d}$ and $C_d=2^d$. As a consequence, if $\mathcal{B}_M(x,r)$ denotes the geodesic ball of radius $r$ centered at $x$, then, if $r \leq \frac{\rho}{4}$,
      \[
      c'_d r^d \leq Vol(\mathcal{B}_M(x,r)) \leq C'_d r^d,
      \]
      with $c'_d = c_d V_d$ and $C'_d = C_d V_d$, where $V_d$ denotes the volume of the unit $d$-dimensional Euclidean ball.
      \end{prop}
      
\begin{proof}[\proofof Proposition \ref{volumeform}]
Denoting $A_{r,v} = \dd_{rv} \exp_x$, the Area Formula \cite[Section 3.2.5]{Federer69} asserts that $J(r,v) = r^{d-1} \sqrt{\det \left( A_{r,v}^t A_{r,v} \right)}$. Note that from Proposition 6.1 in \cite{Niyogi08} together with the Gauss equation \cite[p. 130]{DoCarmo92}, the sectional curvatures in ${M}$ are bounded by $|\kappa| \leq 2/\rho^2$. 
Therefore, the Rauch theorem \cite[Lemma 5]{Dyer15} states that
\[ \left( 1 - \frac{r^2}{3\rho^2} \right) \norm{w} \leq \norm{A_{r,v} w} \leq \left(1+ \frac{r^2}{\rho^2} \right)\norm{w},\]
for all $w \in T_x {M}$. As a consequence, 
\[ 2^{-d} \leq \left( 1 - \frac{r^2}{3\rho^2} \right)^d \leq \sqrt{\det \left( A_{r,v}^t A_{r,v} \right)} \leq \left(1+ \frac{r^2}{\rho^2} \right)^d \leq 2^d. 
\]
Since $Vol(\mathcal{B}_M(x,r))= \int_{s=0}^{r} \int_{v \in {\mathcal{S}}_{d-1}} J(s,v) ds dv$, where ${\mathcal{S}}_{d-1}$ denotes the unit $d-1$-dimensional sphere, the bounds on the volume easily follows.
\end{proof}

\section{Some Technical Properties of the Statistical Model}\label{abstandard}
\subsection{Covering and Mass}
\begin{lem}
\label{epsilon_rate}
Let $Q_0 \in \mathcal{U}_M(f_{min},f_{max})$. Then for all $p \in {M}$ and $ r \leq \rho/4$,
\[ Q_0\bigl(\mathcal{B}(p,r)\bigr) \geq a_{d}f_{min} r^d,\]
where $a_d > 0$.
As a consequence, for $n$ large enough and for all $Q \in \mathcal{G}_{D,d,f_{min},f_{max},\rho,\sigma}$, with probability larger that $1-\left(\frac{1}{n}\right)^{2/d}$,
\[\dd_{H}(M,\X_n) \leq C_{d,f_{min}} \left( \frac{\log n}{n} \right)^{1/d}+ \sigma.\]
Similarly, for $n$ large enough and for all $P \in \mathcal{O}_{D,d,f_{min},f_{max},\rho,\beta}$, with probability larger that $1-\left(\frac{1}{n}\right)^{2/d}$,
\[\dd_{H}(M,\X_n\cap M) \leq C_{d,f_{min}} \left( \frac{\log n}{\beta n} \right)^{1/d}.\]
\end{lem}
\begin{proof}[\proofof Lemma \ref{epsilon_rate}]
The first statement is a direct corollary of Proposition \ref{volumeform}, since for all $r \leq \rho/4$,
\[
Q_0\bigl(\mathcal{B}(p,r)\bigr) 
= 
\int_{\mathcal{B}(p,r)} f d \mathcal{H}^d 
\geq 
f_{min} Vol\left(\mathcal{B}(p,r)\cap M\right)
\geq 
a_d f_{min} r^d
,
\]
where $a_d$ can be taken to be equal to $c'_d$ of Proposition \ref{volumeform}.
Let us now prove the second statement. 
By definition, sample $X_i \in \X_n$, that has distribution $Q \in \mathcal{G}_{D,d,f_{min},f_{max},\rho,\sigma}$ can be written as $X_i = Y_i + Z_i$, with $Y_i$ having distribution $Q_0 \in \mathcal{G}_{D,d,f_{min},f_{max},\rho}$, and $\norm{Z_i} \leq \sigma$.
From the previous point, letting $a = a_d f_{min}$, $Q_0$ fulfils the so-called $(a,d)$-standard assumption of \cite{Chazal2013} for $r\leq \rho/4$. 
Looking carefully at the proof of Lemma 10 in \cite{Chazal2013} shows that its conclusion still holds for measures satisfying the $(a,d)$-standard assumption for small radii only.
Therefore, writing $\Y_n = \left\{Y_1,\ldots,Y_n\right\}$, for $r \leq \rho/8$ we obtain
\[
\p_{Q_0}\left( \dd_{H}({M},\mathbb{Y}_n) > r \right)
\leq \frac{4^d}{a r^d} \exp \left(  -n \frac{a}{2^d} r^d \right).
\]
The statement then follows using that $d_H(\X_n,\Y_n) \leq \sigma$, and setting $r = C_{d,f_{min}} \left( \frac{\log n}{n}\right)^{1/d}$ with $C_{d,f_{min}}^d \frac{a}{2^{d+1}} \geq 1+ 2/d$.

To prove the last point, notice that for all $k=0,\ldots,n$, conditionally on the event $\left\lbrace|\mathbb{X}_n \cap {M}|=k\right\rbrace$, $\mathbb{X}_n \cap {M}$ has the distribution of a $k$-sample of $Q_0$. Therefore,

\begin{align*}
\p_P\left( \dd_{H}({M},\mathbb{X}_n \cap {M}) > r \right| \left. |\mathbb{X}_n \cap {M}|=k  \right)
&=
\p_{Q_0}\left( \dd_{H}({M},\mathbb{X}_k \cap {M}) > r \right)
\\
&\leq
\frac{4^d}{a r^d} \exp \left(  -k \frac{a}{2^d} r^d \right).
\end{align*}
Hence,
\begin{align*}
\p_P\left( \dd_{H}({M},\X_n \cap {M}) > r \right )
&= 
\sum_{k = 0}^{n} 
\p_P\left( \dd_{H}({M},\mathbb{X}_n \cap {M}) > r \right| \left. |\mathbb{X}_n \cap {M}|=k  \right)
\p_P(\left|\mathbb{X}_n \cap {M}\right|=k  ) \\
&\leq
\sum_{k = 0}^{n}
\frac{4^d}{a r^d} \exp \left(  -k \frac{a}{2^d} r^d \right)
\binom{n}{k} \beta^k(1-\beta)^{n-k}\\
&=
\frac{4^d}{a r^d} \left[ 1 - \beta \left( 1 - \exp\left({-\frac{a}{2^d} r^d}\right)  \right) \right]^n \\
&\leq
\frac{4^d}{a r^d} \exp \left[ - n \beta \left( 1 - \exp\left({-\frac{a}{2^d} r^d}\right) \right) \right] \\
&\leq
\frac{4^d}{a r^d} \exp \left[ - n \beta \frac{a}{2^{d+1}}r^d \right],
\end{align*}
whenever $r\leq \rho/8$ and $a r^d \leq 2^d$. Taking $r = C_{d,f_{min}}' \left( \frac{\log n}{\beta n}\right)^{1/d}$ with $C_{d,f_{min}}'^d \frac{\beta a}{2^{d+1}} \geq 1+ 2/d$ yields the result.
\end{proof}
We now focus on proving Lemma \ref{model_properties}.
For its proof, we need the following piece of notation. For all bounded subset $K \subset \R^D$ and $\varepsilon>0$, we let $\mathrm{cv}_K(\varepsilon)$ denote the Euclidean covering number of $K$.
That is, $\mathrm{cv}_K(\varepsilon)$ is the minimal number $k$ of Euclidean open balls of radii $\varepsilon$ and centered at elements of $K$ that are needed to cover $K$.

\begin{lem}\label{covering_number_vs_diameter}
Let $K \subset \R^D$ be a bounded subset. If $K$ is path connected, then for all $\varepsilon>0$, $\mathrm{diam}(K) \leq 2\varepsilon \mathrm{cv}_K(\varepsilon)$.
\end{lem}

\begin{proof}[\proofof Lemma \ref{covering_number_vs_diameter}]
Let $p,q \in K$ and $\gamma : [0,1] \rightarrow K$ be a continuous path joining $\gamma(0) = p$ and $\gamma(1) = q$.
Writing $N = \mathrm{cv}_K(\varepsilon)$, let $x_1,\ldots,x_N \in \R^D$ be the centers of a covering of $K$ by open balls of radii $\varepsilon$.
We let $U_i$ denote $\left\{ t , \norm{\gamma(t) - x_{i}} < \varepsilon \right\} \subset [0,1]$.
By construction of the covering, there exists $x_{(1)} \in \left\{x_1,\ldots,x_N\right\}$ such that $\norm{p - x_{(1)}} < \varepsilon$. 
Then $U_{(1)} \ni \gamma(0) = p$ is a non-empty open subset of $[0,1]$, so that $t_{(1)} = \sup U_{(1)}$ is positive.
If $t_{(1)} =1$, then $\norm{q-x_{(1)}} \leq \varepsilon$, and in particular $\norm{q-p} \leq 2 \varepsilon$.
If $t_{(1)} < 1$, since $U_{(1)}$ is an open subset of $[0,1]$, we see that $\gamma(t_{(1)}) \notin U_{(1)}$. But $\cup_{i = 1}^N U_i$ is an open cover of $[0,1]$, which yields the existence $U_{(2)}$ such that $\gamma(t_{(1)}) \in U_{(2)}$, and for all $t < t_{(1)}$, $\gamma(t) \notin U_{(2)}$. Then consider $t_{(2)} = \sup U_{(2)}$, and so on.
Doing so, we build by induction a sequence of numbers $0 < t_{(1)} < \ldots < t_{(k)} \leq 1$ and distinct centers $x_{(1)},\ldots,x_{(k)} \in \left\{x_1,\ldots,x_N\right\}$  ($k \leq N$) such that
$\norm{p - x_{(1)}} < \varepsilon$, $\norm{q - x_{(k)}} \leq \varepsilon$, 
with 
$\norm{\gamma(t_{(i)}) - x_{(i)}} \leq \varepsilon$ for $1 \leq i \leq k$ and $\norm{\gamma(t_{(i)}) - x_{(i+1)}} < \varepsilon$ for $1 \leq i \leq k-1$.
In particular, $\norm{x_{(i)} - x_{(i+1)}} \leq 2\varepsilon$ for all $1\leq i \leq k-1$.
To conclude, write
\begin{align*}
\norm{p-q}
&\leq 
\norm{p-x_{(1)}}
+
\norm{x_{(1)} - x_{(k)}}
+
\norm{q - x_{(k)}}
\\
&\leq
\varepsilon
+
\sum_{i=1}^{k-1} \norm{x_{(i)} - x_{(i+1)}}
+
\varepsilon
\\
&\leq
2k\varepsilon
\leq
2\varepsilon \mathrm{cv}_K(\varepsilon).
\end{align*}
Since this bound holds for all $p,q \in K$, we get the announced bound on the diameter of $K$.
\end{proof}
We are now in position to prove Lemma \ref{model_properties}.
\begin{proof}[\proofof Lemma \ref{model_properties}]
Let $\varepsilon \leq {\rho}/{4}$, and $x_1,\ldots,x_{\mathrm{cv}_M(\varepsilon)}$ be a minimal covering of $M$.
According to Lemma \ref{epsilon_rate}, for all $k$,
\[Q\left(\mathcal{B}_M(x_k,\varepsilon)\right) \geq a_d f_{min} \varepsilon^d\]
for some $a_d > 0$. A straightforward packing argument \cite[Section B.1]{Chazal2013} yields that the covering number of the support $M$ of $Q$ satisfies
\[
\mathrm{cv}_M(\varepsilon) \leq \frac{c_d}{f_{min} \varepsilon^{d}}
\] 
for all $\varepsilon \leq {\rho}/{4}$, where $c_d = 2^d/a_d$.
Applying this bound with $\varepsilon = \rho/4$, together with Lemma \ref{covering_number_vs_diameter} yields
\begin{align*}
\mathrm{diam}(M)
&\leq
2 \frac{\rho}{4} \mathrm{cv}_M\left( \frac{\rho}{4} \right)
\\
&\leq
\frac{\rho}{2} \frac{c_d}{f_{min} \left(\frac{\rho}{4}\right)^d}
\\
&=
\frac{C_d}{f_{min} \rho^{d-1}},
\end{align*}
where $C_d = 2^{3d-1}/a_d$.
\end{proof}

Now we allow for some outliers. We consider a random variable $X$ with distribution $P$, that can be written as $X = V(Y+Z) + (1-V)X''$, with $\|Z\| \leq s h$, $s \leq 1/4$, such that $\mathbb{P}(V=1) = \beta$ and $V$ is independent from $(Y,Z,X'')$, $Y$ has law $Q$ in $\mathcal{G}_{D,d,f_{min},f_{max},\rho}$, and $X''$ has uniform distribution on $\mathcal{B}(0,K_0)$ (recall that $K_0$ is defined below Lemma \ref{model_properties}). Note that $s=0$ corresponds to the clutter noise case, whereas $\beta=1$ corresponds to the additive noise case.

  For a fixed point $x$, let $p\left(x,h\right)$ denote $P\left(\mathcal{B}\left(x,h\right)\right)$. We have $\mathbb{P}\left(VY \in \mathcal{B}\left(x,\left(1-s\right)h\right) \right) \leq \mathbb{P}\left(VX \in \mathcal{B}\left(x,h\right)\right)\leq \mathbb{P}\left(VY \in \mathcal{B}\left(x,2h\right)\right)$. Hence we may write 
\[
 \beta q\left(x,3/4h\right) + \left(1-\beta\right)q'\left(x,h\right) \leq p\left(x,h\right) \leq  \beta q\left(x,2h\right) + \left(1-\beta\right)q'\left(x,h\right),
 \]
  where $q\left(x,h\right) = Q\left(\mathcal{B}\left(x,h\right)\right)$, and {\color{black}{$q'\left(x,h\right)$ = $\left(h/{K_0}\right)^D$}}. Bounds on the quantities above are to be found in the following lemma.
     \begin{lem}\label{enoughweight}
     There exists $h_+\left(\rho,\beta,f_{min},f_{max},d\right) \leq \rho / \sqrt{12d}$ such that, if $h \leq h_+$, for every $x$ such that $\dd\left(x,M\right) \leq h$, we have 
     \begin{itemize}
     \item $\mathcal{B}\left(x,2h\right) \cap M \subset \mathcal{B}\left(\pi_{{M}}\left(x\right),4h\right) \cap M,$
     \item $ q\left(x,2h\right) \leq C_{d}{ f_{max}} h^{d}$.
     \end{itemize}
     Moreover, if $\dd\left(x,M\right) \leq h/\sqrt{2}$, we have
     \begin{itemize}
     \item $\mathcal{B}\left(\pi_{{M}}\left(x\right),h/8\right) \cap M \subset \mathcal{B}\left(x,3h/4\right)$,
     \item $c_{d}{ f_{min}} h^d \leq q\left(x,3h/4\right)$,
     \item $ p\left(x,h\right) \leq 2 \beta q\left(x,2h\right)$.
     \end{itemize}
     \end{lem}
     \begin{proof}[\proofof Lemma \ref{enoughweight}]
     Set $h_1\left(\rho\right) = \rho/\left(16 \alpha\right)$, and let $x$ be such that $\dd\left(x,{M}\right) \leq h$, and $h \leq h_1$.  According to Proposition \ref{ballprojection},   $\mathcal{B}\left(x,2h\right) \cap {M} \subset \mathcal{B}\left(\pi_{{M}}\left(x\right), r_{2h}^+\right) \cap M$, with $r_{2h}^+ = \sqrt{\left(1 + 2 \Delta / \rho\right)} r_{2h} \leq 2 r_{2h} \leq 4h$. According to Proposition \ref{geodesics}, if $y$ $\in$ $\mathcal{B}\left(\pi_M\left(x\right),4h\right) \cap M$, then $d_M\left(\pi_M\left(x\right),y\right) \leq 4\alpha h \leq \rho/4$.   Proposition \ref{volumeform} then yields $q\left(x,2h\right) \leq C_d f_{max}h^d$.
     
      Now if $\dd\left(x,M\right) \leq h/\sqrt{2}$, $\mathcal{B}\left(\pi_{{M}}\left(x\right),r_{3h/4}^-\right) \cap M \subset \mathcal{B}\left(x,3h/4\right) \cap M$ from Proposition \ref{ballprojection}, with $r_{3h/4}^- = \sqrt{\left(1 - \Delta / \rho\right)}r_{3h/4} \geq r_{3h/4}/2 \geq h/8$. Since $\mathcal{B}_M\left(\pi_M\left(x\right),h/8\right) \subset \mathcal{B}\left(\pi_M\left(x\right),h/8\right)\cap M$, a direct application of Proposition \ref{volumeform} entails $c_{d}{ f_{min}} h^d \leq q\left(x,3h/4\right)$. 
      
      Applying Proposition \ref{volumeform} again, there exists $h_2\left(f_{min},d,D,\beta,\rho\right)$ such that if $h \leq h_1 \wedge h_2$, then for any $x$ such that $\dd\left(x,{M}\right) \leq h/\sqrt{2}$ we have $\left(1-\beta\right) q'\left(x,h\right) \leq \beta c_{d,f_{min}} h^{d} $, along with $q\left(x,2h\right) \geq q\left(x,3h/4\right) \geq c_{d,f_{min}} h^{d}$. We deduce that $ p\left(x,h\right) \leq 2 \beta q\left(x,2h\right)$. Taking $h_+ = h_1 \wedge h_2\wedge \rho / \sqrt{12d}$ leads to the result.
     \end{proof}

\subsection{Local Covariance Matrices}
     In this section we describe the shape of the local covariance matrices involved in tangent space estimation. Without loss of generality, the analysis will be conducted for $\hat{\Sigma}_1$ (at sample point $X_1$), abbreviated as $\hat{\Sigma}$. We further assume that $d(X_1,M) \leq h/\sqrt{2}$, $\pi_{{M}}(X_1) = 0$, and that $T_{0}{M}$ is spanned by the $d$ first vectors of the canonical basis of $\mathbb{R}^D$.
     
The two models (additive noise and clutter noise) will be treated jointly, by considering a random variable $X$ of the form
\[ 
X = V(Y+Z) + (1-V)X'',
\]
where $\mathbb{P}(V=1) = \beta$ and $V$ is independent from $(Y,Z,X'')$, $Y$ has distribution in $\mathcal{G}_{D,d,f_{min},f_{max},\rho,\sigma}$, $\|Z\| \leq \sigma$, and $X''$ has uniform law on $\mathcal{B}(0,K_0)$ (recall that $K_0$ is defined above Definition \ref{noise_model_definition}). For short we denote by $s$ the quantity $\sigma/h$, and recall that we take $s \leq 1/4$, along with $h \leq h_+$ (defined in Lemma \ref{enoughweight}).

Let $U(X_i,h)$, $i=2, \hdots, n$, denote $ \mathbbm{1}_{\mathcal{B}\left (X_1,h \right )}(X_i)$, let $Y_i \in M$ and $Z_i$ such that $X_i = Y_i + Z_i$, with $\|Z_i\| \leq s h$, and let $V_2, \hdots V_n$ denote random variables such that $V_i= 1$ if $X_i$ is drawn from the signal distribution (see page \pageref{labels}). It is immediate that the $(U(X_i,h),V_i)$'s are independent and identically distributed, with distribution $(U(X,h),V)$.

     With a slight abuse of notation, we will denote by $\mathbb{P}$ and $\mathbb{E}$ conditional probability and expectation with respect to $X_1$. The following expectations will be of particular interest.        
     \[
     \begin{array}{@{}ccc}
     m(h) &=& \mathbb{E}(XU(X,h)V)/\mathbb{E}(VU(X,h)), \\
     \Sigma(h) &=& \mathbb{E} (X-m(h))_{\top}(X-m(h))_{\top}^{t}U(X,h)V,
     \end{array}
     \]
     where for any $x$ in $\mathbb{R}^D$ $x_{\top}$ and $x_{\perp}$ denote respectively the projection of $x$ onto $T_0 M$ and $T_0M^{\perp}$.
     
    The following lemma gives useful results on both $m(h)$ and $\Sigma(h)$, provided that $X_1$ is close enough to $M$. 
    
       \begin{lem}\label{expectations}
       If $d(X_1,M) \leq h/\sqrt{2}$, for $h \leq h_+$, then
	\[
	\Sigma(h) = \left ( \begin{array}{@{}c|c}
      A(h) & 0 \\
      \hline
      0 & 0
      \end{array}
      \right ),
	\]
	with 
	\[
	\mu_{min}(A(h)) \geq \beta {c_{d,f_{min},f_{max} }}h^{d+2}.
	\]
	Furthermore, 
	\begin{align*}
	\|m_{\top}(h)\| &\leq 2h, \\
	\|m_{\perp}(h)\| &\leq \frac{2h^2}{\rho} + s h.
	\end{align*}
	\end{lem}
	\begin{proof}[\proofof Lemma \ref{expectations}]
	Let $x=y+z$ be in $B(X_1,h)$, with $y \in M$ and $\|z\| \leq s h$. Since $s \leq 1/4$, $\|y\| \leq 2h$. According to Proposition {\color{black}{\ref{ballprojection}}} combined with Proposition \ref{geodesics}, we may write, for $h \leq h_+$ and $y$ in $\mathcal{B}(X_1 ,2h) \cap {M}$,
      \[
      y = r v + R(r,v), 
      \]
      in local polar coordinates. Moreover, if $y \in \mathcal{B}(X_1,(1-s)h)$, then $x \in \mathcal{B}(X_1,h)$. Then, according to Proposition \ref{ballprojection}, we have  $\mathcal{B}(\pi_{{M}}(X_1),r_{3h/4}^-) \cap M \subset \mathcal{B}(X_1,(1-s)h) \cap M$. Let $u$ be a unit vector in $T_0 {M}$. Then $\left\langle u,x-m_{\top}(h) \right\rangle^2$ $=$ $\left\langle u,rv + R(r,v)+z-m_{\top}(h) \right\rangle^2$ $\geq $ $\left\langle u,rv -m_{\top}(h) \right\rangle^2 /2 - 3 (R(r,v)+z)^2$ $\geq$ $\left\langle u,rv -m_{\top}(h) \right\rangle^2 /2-6 r^4/(4\rho^2) - 6 s^2 h^2$ according to Proposition \ref{geodesics}. Hence we may write
      \begin{align*}
      \left\langle Au,u \right\rangle & =  \beta \int_{\mathcal{B}(X_1,h) \cap {M}} \left\langle u,rv + R(r,v)-m_{\top}(h) \right\rangle^2 J(r,v){f(r,v)}drdv \\
      & \geq \beta f_{min}  c_d \int_{r=0}^{r_{3h/4}^-} \int_{{\mathcal{S}}_{d-1}}{r^{d-1}\left [\left\langle u,rv-m_\top (h)\right\rangle^2 /2 - 3 r^4/(2\rho^2) -6 s^2 h^2 \right ]drdv},
      \end{align*}
      	according to Proposition \ref{volumeform} (bound on $J(r,v)$) and Proposition \ref{ballprojection} (the geodesic ball $\mathcal{B}_M(\pi_M(X_1),r_{3h/4}^-)$ is included in the Euclidean ball $\mathcal{B}(\pi_M(X_1),r_{3h/4}^-)$ $\subset$ $\mathcal{B}(X_1,(1-s)h) \cap {M}$). Then
      	\begin{align*}
      	\int_{r=0}^{r_{3h/4}^-}\int_{{\mathcal{S}}_{d-1}}{\frac{r^{d-1}\left\langle u,rv-m_\top (h)\right\rangle^2}{2} drdv} &\geq \int_{r=0}^{r_{3h/4}^-}\int_{{\mathcal{S}}_{d-1}}{\frac{r^{d-1}\left\langle u,rv\right\rangle^2}{2} drdv} \\ &= \frac{\sigma_{d-1}}{2d} \int_{r=0}^{r_{3h/4}^-} r^{d+1} dr \\ &= \frac{\sigma_{d-1}(r_{3h/4}^-)^{d+2}}{2d(d+2)}, 
      	\end{align*}
        where $\sigma_{d-1}$ denotes the surface of the $d-1$-dimensional unit sphere. On the other hand,
        \begin{align*}
        \int_{r=0}^{r_{3h/4}^-} \int_{{\mathcal{S}}_{d-1}}{\frac{3 r^{d+3}}{2\rho^2} + 6 s^2 h^2r^{d-1}drdv} = \sigma_{d-1}(r_{3h/4}^-)^{d+2} \left ( \frac{3 (r_{3h/4}^-)^{2}}{ 2(d+4)\rho^2} + \frac{6 s^2h^2}{d}\right ).
        \end{align*}
        Since $r_{3h/4}^- \leq h \leq h_+ \leq \rho/\sqrt{12d}$, we conclude that
        \begin{align*}
       \left\langle Au,u \right\rangle \geq \beta c_d f_{min} (r_{3h/4}^-)^{d+2} \geq \beta c_d f_{min} h^{d+2},
\end{align*}         	
      	since, for $d(X_1,M) \leq h/\sqrt{2}$ and $h \leq h_+$, $r_{3h/4}^- \geq r_{3h/4}/2 \geq h/8$, according to Proposition \ref{ballprojection}.
      	
      	Now, since for any $x=y+z \in \mathcal{B}(X_1,h)$, $y$ $\in$ $M \cap \mathcal{B}(0,2h)$ and $\|z\| \leq s h$, we have $\|y_{\perp}\| \leq 2h^2/\rho$, according to Proposition \ref{normalpart}. Jensen's inequality yields that $\|m(h)_{\perp}\| \leq 2h^2/\rho + s h$ and $\|m(h)_{\top}\| \leq \|m(h)\| \leq 2h$.
       \end{proof}

      The following Lemma \ref{Concentration} is devoted to quantify the deviations of empirical quantities such as local covariance matrices, means and number of points within balls from their deterministic counterparts. To this aim we define $N_0(h)$ and $N_1(h)$ as the number of points drawn from  respectively noise and signal in $\mathcal{B}(X_1,h) \cap M$, namely
     \begin{align*}
     N_0(h) &= \sum_{i \geq 2} U(X_i,h) (1-V_i), \\
     N_1(h) &= \sum_{i \geq 2} U(X_i,h) V_i.
     \end{align*}

	\begin{lem}\label{Concentration}
Recall that $h_0 = \left ( \kappa \frac{\log n}{\beta(n-1)} \right )^{1/(d+1)}$ (as defined page \pageref{def_h0}), and $h_{\infty} = h_0^{(d+1)/d}$, for $\kappa$ to be fixed later. 

If $h_0 \leq h_+$ and $d(X_1, M) \leq h_+ /\sqrt{2}$, then, with probability larger than $1- 4 \left(\frac{1}{n}\right)^{2/d +1}$, the following inequalities hold, for all $h \leq h_0$.
\[
\begin{array}{@{}ccc}
   \frac{N_0(h)}{n-1} & \leq &2 (1- \beta) q'(h) + \frac{10(2+2/d)\log n}{n-1},     \\
   \frac{N_1(h)}{n-1} & \leq &2 \beta q(2h) +  \frac{10(2+2/d)\log n}{n-1}. 
\end{array}
\]
Moreover, for all $(h_{\infty} \vee \sqrt{2} \dd(X_1,M)) \leq h \leq h_0$, and $n$ large enough,
      \begin{align*}
      \left \| \frac{1}{n-1}\sum_{i \geq 2} (X_i-m(h))_{\top}(X_i-m(h))_{\top}^{t}U(X_i,h)V_i - \Sigma(h) \right \|_\mathcal{F} &\leq C_d \frac{f_{max}}{f_{min}\sqrt{\kappa}} \beta q(2h) h^2 , \\
      \frac{1}{n-1}\left \| \sum_{i \geq 2}{(X_i - m(h))_{\top}U(X_i,h)V_i} \right \|_{\mathcal{F}} &\leq C_d \frac{f_{max}}{f_{min}\sqrt{\kappa}} \beta q(2h) h. 
      \end{align*}
\end{lem}  	
\begin{proof}[\proofof Lemma \ref{Concentration}]
The first two inequalities are straightforward applications of Theorem 5.1 in \cite{Boucheron05}. The proofs of the two last results are detailed below. They are based on Talagrand-Bousquet's inequality (see, e.g., Theorem 2.3 in \cite{Bousquet02}) combined with the so-called peeling device.

Define $h_-=(h_{\infty} \vee \sqrt{2} \dd(X_1,M))$, where we recall that in this analysis $X_1$ is fixed,
and let $f_{T,h}$ denote the function
\[
f_{T,h}(x,v) = \left\langle T,(x-m(h))_{\top}(x-m(h))_{\top}^{t} U(x,h)v \right\rangle,
\]
for $h_- \leq h \leq h_0$, $T$ a $d \times d$ matrix such that $\|T\|_{\mathcal{F}}=1$, $x$ in $\mathbb{R}^D$, $v$ in $\{0,1\}$, and $\left\langle T, B \right\rangle = \mathrm{trace}(T^t A)$, for any square matrices $T$ and $A$. Now we define the weighted empirical process
\[
Z = \sup_{T,h} \hspace{1ex} {\sum_{i \geq 2}{\frac{f_{T,h}(X_i,V_i) - \mathbb{E}f_{T,h}(X,V)}{r(h)}}},
\]
with $r(h) = \beta q(2h) h^2$, along with the constrained empirical processes 
\[
Z(u) = \sup_{T,h \leq u} \hspace{1ex} {\sum_{i \geq 2}{{f_{T,h}(X_i,V_i) - \mathbb{E}f_{T,h}(X,V)}}}, 
 \]
for $h_- \leq u \leq h_0$. Since $\| f_{T,h} \|_{\infty} \leq sup_{x \in M} \|x-m(h)\|^2 U(x,h) \leq 4h^2$, and 
 \[
 Var(f_{T,h}(X,V)) \leq \mathbb{E} \left ( \| X-m(h)\|^2 U(X,h) V \right ) \leq 16 \beta h^4 \mathbb{P}(VX \in \mathcal{B}(X_1,h) \leq 16 \beta h^4 \mathbb{P}(VY \in \mathcal{B}(X_1,2h),
 \]
 for $s \leq 1/4$, a direct application of Theorem 2.3 in \cite{Bousquet02} yields, with probability larger than $1-e^{-x}$,
  \[
  Z(u) \leq 3 \mathbb{E} Z(u) + \sqrt{\frac{32 \beta q(2u) u^4x}{n-1}} + \frac{20u^2x}{3(n-1)}.
  \]
  To get a bound on $\mathbb{E} Z(u)$, we introduce some independent Rademacher random variables $\sigma_2, \hdots, \sigma _n$, i.e. $\mathbb{P}(\sigma_j =1) = \mathbb{P}(\sigma_j=-1)=1/2$. With a slight abuse of notation, expectations with respect to the $(X_i,V_i)$'s and $\sigma_i$'s, $i=2, \hdots, n$, will be denoted by $\mathbb{E}_{(X,V)}$ and $\mathbb{E}_{\sigma}$ in what follows. According to the symmetrization principle (see, e.g., Lemma 11.4 in \cite{Massart13}), we have
  \begin{align*}
(n-1)\mathbb{E}Z(u) &\leq 2 \mathbb{E}_{(X,V)} \mathbb{E}_{\sigma_i} \sup _{h \leq u,T}  \sum_{i \geq 2} \left\langle T,\sigma_i V_i U (X_i,h) ((X_i - m(h))_{\top}(X_i - m(h))_{\top}^t) \right\rangle  \\
& \begin{multlined}[t] \leq 2 \mathbb{E}_{(X,V)} \mathbb{E}_{\sigma}  \sup_{h \leq u,T}   \sum_{i \geq 2} \sigma_i \left\langle V_i U(X_i,h) X_i X_i^t,T \right\rangle \\ +  2 \mathbb{E}_{(X,V)} \mathbb{E}_{\sigma}  \sup_{h \leq u,T}   \sum_{i \geq 2} \sigma_i \left\langle V_i U(X_i,h) X_i m(h)^t,T \right\rangle \\ +  2 \mathbb{E}_{(X,V)} \mathbb{E}_{\sigma}  \sup_{h \leq u,T}   \sum_{i \geq 2} \sigma_i \left\langle V_i U(X_i,h) m(h) X_i^t,T \right\rangle \\ + 2 \mathbb{E}_{(X,V)} \mathbb{E}_{\sigma}  \sup_{h \leq u,T}   \sum_{i \geq 2} \sigma_i \left\langle V_i U(X_i,h) m(h) m(h)^t,T \right\rangle
\end{multlined} \\
& := 2\mathbb{E}_{(X,V)}(E_1 + E_2 + E_3 + E_4).
\end{align*}
    For a fixed sequence $(X_i,V_i)$, $i=2, \hdots, n$, we may write
    \begin{align*}
    E_1 &\leq \begin{multlined}[t] 
    			\mathbb{E}_\sigma \sup_{h \leq u} \left ( \left \| \sum_{i \geq 2} \sigma_i V_i U(X_i,h) X_i X_i^t \right \|_{\mathcal{F}} - \mathbb{E}_{\sigma} \left \| \sum_{i \geq 2} \sigma_i V_i U(X_i,h) X_i X_i^t \right \|_{\mathcal{F}} \right ) 
    			\\+ \sup_{h \leq u} \mathbb{E}_{\sigma} \left \| \sum_{i \geq 2} \sigma_i V_i U(X_i,h) X_i X_i^t \right \|_{\mathcal{F}} 
    			\end{multlined}
    			 \\
     & := E_{11} + E_{12}.
\end{align*}     
Jensen's inequality ensures that
    \begin{align*}
E_{12} & \leq \sup_{h \leq u}\sqrt{\mathbb{E}_\sigma \left \| \sum_{i \geq 2} \sigma_i V_i U(X_i,h) X_i X_i^t \right \|_{\mathcal{F}}^2} \\
                   & \leq 4u^2\sqrt{N_1(u)},
\end{align*} 
hence
\[
\mathbb{E}_{(X,V)}  E_{12} \leq 4 u^2 \sqrt{\beta (n-1)q(2u)}.
\]
For the term $E_{11}$, note that, when $(X_i,V_i)_{i=2, \hdots,n}$ is fixed,  $$\sup_{h \leq u} \left ( \left \| \sum_{i \geq 2} \sigma_i V_i U(X_i,h) X_i X_i^t \right \|_{\mathcal{F}} - \mathbb{E}_{\sigma} \left \| \sum_{i \geq 2} \sigma_i V_i U(X_i,h) X_i X_i^t \right \|_{\mathcal{F}} \right )$$ is in fact a supremum of at most $N_1(u)$ processes. According to the bounded difference inequality (see, e.g., Theorem 6.2 of \cite{Massart13}), each of these processes is subGaussian with variance bounded by $16h^4N_1(u)$ (see Theorem 2.1 of \cite{Massart13}). Hence a maximal inequality for subGaussian random variables (see Section 2.5, p.31, of \cite{Massart13}) ensures that 

      \[
      E_{11} \leq 4 h^2 \sqrt{2 N_1(u) \log(N_1(u))} \leq 4 h^2 \sqrt{2 N_1(u) \log(n-1)}.
      \]
      Hence $\mathbb{E}_{(X,V)} E_{11} \leq 4h^2\sqrt{2 \beta (n-1) q(2u) \log(n-1)}$.
      $E_2$ may also be decomposed as
      \begin{align*}
      E_2 & = \mathbb{E}_{\sigma} \sup_{h \leq u} \left \| \left ( \sum_{i\geq 2} \sigma_i V_i U(X_i,h) X_i \right )m(h)^t \right \|_{\mathcal{F}} \\
         & \leq 2u  \mathbb{E}_{\sigma} \sup_{h \leq u} \left \| \sum_{i\geq 2} \sigma_i V_i U(X_i,h) X_i  \right \| \\
         & \leq 2u \left (\mathbb{E}_{\sigma} \sup_{h \leq u} \left (  \left \| \sum_{i\geq 2} \sigma_i V_i U(X_i,h) X_i  \right \| - \mathbb{E}_{\sigma}  \left \| \sum_{i\geq 2} \sigma_i V_i U(X_i,h) X_i  \right \| \right ) + \sup_{h \leq u} \mathbb{E}_{\sigma} \left \| \sum_{i\geq 2} \sigma_i V_i U(X_i,h) X_i  \right \| \right ) \\
         &:= 2u( E_{21} + E_{22} ).
\end{align*}             
   Jensen's inequality yields that $E_{22} \leq 2u \sqrt{N_1(u)}$, and the same argument as for $E_{11}$ (expectation of a supremum of $n-1$ subGaussian processes
with variance bounded by $4u^2N_1(u)$) gives $E_{22} \leq 2u \sqrt{2 N_1(u) \log(n-1)}$. Hence
\[
\mathbb{E}_{(X,V)} E_2 \leq 4u^2 \sqrt{\beta (n-1) q(2u)} \left ( \sqrt{2 \log(n-1)} +1 \right ).
\]
Similarly, we may write
\[
\mathbb{E}_{(X,V)} E_3 \leq 4u^2 \sqrt{\beta (n-1) q(u)} \left ( \sqrt{2 \log(n-1)} +1 \right ).
\]
At last, we may decompose $E_4$ as
\begin{align*}
E_4 & \leq \mathbb{E}_{\sigma} 4u^2 \sup_{h \leq u} \left | \sum_{i \geq 2} V_i U(X_i,h) \right | \\
    & \leq 4u^2\left [\mathbb{E}_{\sigma} \sup_{h \leq u} \left (\left | \sum_{i \geq 2} V_i U(X_i,h) \right | - \mathbb{E}_{\sigma}\left | \sum_{i \geq 2} V_i U(X_i,h) \right | \right ) + \sup_{h \leq u} \mathbb{E}_{\sigma}\left | \sum_{i \geq 2} V_i U(X_i,h) \right | \right ] \\
    & \leq 4u^2\sqrt{N_1(u)} \left ( \sqrt{2 \log(n-1)} +1 \right ),
\end{align*}
using the same argument. Combining all these terms leads to
\[
\mathbb{E} Z(u) \leq \frac{32 \sqrt{\beta q(2u)}}{\sqrt{n-1}} \left ( \sqrt{2 \log(n-1)} +1 \right ), 
\]
hence we get
\[
\mathbb{P} \left ( Z(u) \geq \frac{192 \sqrt{2}u^2 \sqrt{ \beta q(2u)\log(n-1)}}{\sqrt{n-1}} \left ( 1 + \frac{1}{48}\sqrt{\frac{x}{\log(n-1)}} \right ) + \frac{20 u^2 x}{n-1} \right ) \leq e^{-x}.
\]
To derive a bound on the weighted process $Z$, we make use of the so-called peeling device (see, e.g., Section 13.7, p.387, of \cite{Massart13}).
Set $p = \lceil \log(h_0/h_\infty) \rceil \leq 1 + \log(h_0/h_\infty)$, so that $e^{-{p}}h_0 \leq h_-$. According to Lemma \ref{enoughweight}, if $I_j$ denotes the slice $[e^{-j}h_0, e^{-(j-1)}h_0] \cap [h_-,h_0]$, then, for every $h$ in $I_j$, we have
\[
r(h) \geq r(h_{j-1}) c_d \frac{f_{min}}{f_{max}},
\]
where $c_d$ depends only on the dimension, provided that $h_0 \leq h_+$. Now we may write
\begin{align*}
\mathbb{P} &\left ( Z \geq \frac{192f_{max}\sqrt{2}}{f_{min}c_d \sqrt{ \beta q(2h_-)(n-1)}} \left ( 1 + \frac{1}{48} \sqrt{\frac{x + \log(p)}{n-1}} \right ) + \frac{20 f_{max} (x + \log(p)) }{(n-1) \beta c_{d} f_{min} q(2h_-)} \right ) 
\\
&\leq 
\sum_{j=1}^{p}\mathbb{P} \left ( \sup_{T,h \in I_j} \frac{\sum_{i\geq 2} f_{T,h}(X_i,V_i) - \mathbb{E}f_{T,h}(X,V)}{r(h)}
\right.
\\
&\hspace{20ex} \left. \geq \frac{192f_{max}\sqrt{2}}{f_{min}c_d \sqrt{ \beta q(2h_-)(n-1)}} \left ( 1 + \frac{1}{48}\sqrt{\frac{x + \log(p)}{n-1}} \right ) + \frac{20 f_{max} (x + \log(p)) }{(n-1) f_{min}c_d \beta q(2h_-)} \right ) 
\\
&\leq \sum_{j=1}^{p}{ \mathbb{P} \left ( Z(h_{j-1}) \geq \frac{192 \sqrt{2}r(h_{j-1})}{\sqrt{ \beta q(2h_-)(n-1)}}\left ( 1 + \frac{1}{48}\sqrt{\frac{x + \log(p)}{n-1}} \right ) + \frac{20 r(h_{j-1}) (x + \log(p)) }{(n-1)  \beta q(2h_-)} \right ) }. 
\end{align*}
Since $q(2h_{j-1}) \geq q(2h_-)$, we deduce that 
\begin{align*}
\mathbb{P} \left ( Z \geq \frac{192f_{max}\sqrt{2}}{f_{min}c_d \sqrt{\beta q(2h_-)(n-1)}} \left ( 1 + \frac{1}{48}\sqrt{\frac{x + \log(p)}{n-1}} \right ) + \frac{20 f_{max} (x + \log(p)) }{(n-1) c_{d} f_{min} \beta q(2h_-)} \right )  & \leq p e^{ -(x+ \log(p))} \\
& = e^{-x}.
\end{align*}
Now, according to Lemma \ref{enoughweight}, $\beta q(2h_-) \geq c_d \kappa \log n /(n-1)$. On the other hand, $p \leq 1 + \log(h_0/h_\infty) \leq \log(\beta(n-1)/\kappa)/d \leq \log n  /d$, for $\kappa \geq 1$. For $n$ large enough, taking $x = \left ( 1 + 2/d \right ) \log n$ in the previous inequality, we get
\[
\mathbb{P} \left(Z \geq C_d \frac{f_{max}}{f_{min} \sqrt{\kappa}} \right ) \leq \left ( \frac{1}{n} \right )^{1 + 2/d}.
\]
The last concentration inequality of Lemma \ref{Concentration} may be derived the same way, considering the functions
\[
g_{T,h}(x,v)  = \left\langle (x-m(h))U(x,h)v,T\right\rangle,
\]
where $T$ is an element of $\mathbb{R}^d$ satisfying $\|T\| \leq 1$.
\end{proof}

\subsection{Decluttering Rate}
 In this section we prove that, if the angle between tangent spaces is of order $h$, then we can distinguish between outliers and signal at order $h^2$. We recall that the slab $S(x,T,h)$ is the set of points $y$ such that $ \| \pi_{T}(y-x) \| \leq k_1 h$ and $\| \pi_{T^{\perp}}(y-x)\| \leq k_2h^2$, $k_1$ and $k_2$ defined in Lemma \ref{slabs}, and where $\pi_{T}$ denotes the orthogonal projection onto $T$.     
     
     \begin{lem}\label{universalthreshold}
        Recall that $h_0 = \left ( \kappa \frac{\log n}{\beta(n-1)} \right )^{1/(d+1)}$, and $h_{\infty} = h_0^{(d+1)/d}$. Let $K$ be fixed, and $k_1$, $k_2$ defined accordingly from Lemma \ref{slabs}. If $h_0 \leq h_+$, for $\kappa$ large enough (depending on $d$, $\rho$ and $f_{min}$) and $n$ large enough, there exists a threshold $t$ such that, for all $h_{\infty} \leq h \leq h_0$, we have, with probability larger than $1-3\left(\frac{1}{n}\right)^{2/d +1}$,        
        \[
        \begin{array}{@{}ccc}
         X_1 \in M \quad \mbox{and} \quad \angle \left(T,T_{X_1}M \right) \leq K h/\rho & \Rightarrow & |S(X_1,T,h) \cap \{ X_2, \hdots, X_n \}| \geq t(n-1)h^d, \\
        d(X_1,M) \geq h^2/\rho \quad \mbox{and} \quad \angle \left(T,T_{\pi (X_1)}M\right)  \leq K h/\rho & \Rightarrow & |S(X_1,T,h) \cap \{ X_2, \hdots, X_n \}| < t(n-1)h^d, \\
        d(X_1,M) \geq h/\sqrt{2} & \Rightarrow & |S(X_1,T,h) \cap \{ X_2, \hdots, X_n \}| < t(n-1)h^d.
        \end{array}
        \]
        \end{lem}
        
        \begin{proof}[\proofof Lemma \ref{universalthreshold}]
        Suppose that $d(X_1,M) \geq h/\sqrt{2} $. Then, according to Lemma \ref{slabs}, $S(X_1,T,h) \subset \mathcal{B}(X_1,h/2)$, with $\mathcal{B}(X_1, h/2) \cap M = \emptyset$, hence $P_n(S(X_1,T,h)) \leq P_n(\mathcal{B}(X_1,h/2))$. Theorem 5.1 in \cite{Boucheron05} yields that, for all $h_{\infty} \leq h \leq h_0$, with probability larger than $1-\left(\frac{1}{n}\right)^{2/d +1}$,
        \[
        P_n(\mathcal{B}(X_1, h/2)) \leq 2 P(\mathcal{B}(X_1, h/2)) + \frac{4 \left (2/d +1 \right ) \log(8n)}{n-1}.
        \]
        Since $\log(n)/(n-1) \leq \beta h^d/\kappa$, we may write
        \begin{align*}
        P_n(S(X_1,T,h)) & \leq 2 Q'(\mathcal{B}(X_1, h/2)) + \frac{4 \left (2/d +1 \right ) \log(8n)}{n-1} 
        \leq 2(1-\beta) \frac{h^D}{(2K_0)^D} + \frac{4 \left (2/d +1 \right ) \log(8n)}{n-1} \\
        &  \leq (1-\beta)C_{d,D,\rho,f_{min}} h^{d+1} + \frac{4 \left (2/d +1 \right ) \log(8n)}{n-1} \leq  h^d \left ( (1-\beta)C_{d,D,\rho,f_{min}} h + \frac{C_d \beta}{\kappa} \right ),
        \end{align*}
        for $n$ large enough so that $h \leq 1$.
        
        If $h/ \sqrt{2} \geq d(X_1,M) \geq h^2/\rho$ and $\angle \left (T_{\pi(X_1)} M, T \right ) \leq K h / \rho$, then Lemma \ref{slabs} provides a big slab $S'(x,T_{\pi(x)}M,h)$ so that 
$S(x,T,h) \subset S'(x,T_{\pi(x)}M,h)$ and $S'(x,T_{\pi(x)}M,h) \cap M = \emptyset$. Thus, $P_n(S(x,T,h)) \leq P_n(S'(x,T_{\pi(x)}M,h))$. An other application of Theorem 5.1 in \cite{Boucheron05} yields that, for all $h_{\infty} \leq h \leq h_0$, with probability larger than $1-\left(\frac{1}{n}\right)^{2/d +1}$,
       \[
       P_n(S'(x,T_{\pi(x)}M,h)) \leq 2 P(S'(x,T_{\pi(x)}M,h)) + \frac{4 \left (2/d +1 \right ) \log(8n)}{n-1}, 
       \]
       hence, denoting by $\omega_{r}$ the volume of the $r$-dimensional unit ball, we get
       \begin{align*}
       P_n(S(X_1,T,h)) &\leq 2 Q'(\mathcal{B}(X_1, h/2)) + \frac{4 \left (2/d +1 \right ) \log(8n)}{n-1} 
       \\
       &\leq \frac{2(1-\beta)\omega_d \omega_{D-d}}{K_0^D \omega_D}(k'_1h)^{d}(k'_2h^2)^{D-d} + \frac{4 \left (2/d +1 \right ) \log(8n)}{n-1} \\
             &\leq (1-\beta)C_{d,D,f_{min},\rho}h^{d+1} + \frac{4 \left (2/d +1 \right ) \log(8n)}{n-1} 
             \\
             &\leq  h^d \left ( (1-\beta)C_{d,D,\rho,f_{min}} h + \frac{C_d \beta}{\kappa} \right ),
\end{align*}             
           when $n$ is large enough.
           
           Now, if $X_1 \in M$ and $\angle \left (T_{\pi(X_1)} M, T \right ) \leq K h / \rho$, Lemma \ref{slabs} entails that $\mathcal{B}(X_1,k_3h) \cap M \subset S(X_1,T,h)$, hence $P_n(S(X_1,T,h)) \geq P_n(\mathcal{B}(X_1,k_3h) \cap M)$. A last application of Theorem 5.1 in \cite{Boucheron05} yields that, for all $h_{\infty} \leq h \leq h_0$, with probability larger than $1-\left(\frac{1}{n}\right)^{2/d +1}$,
           \[
           P_n(\mathcal{B}(X_1,k_3h) \cap M) \geq \frac{1}{2} P(\mathcal{B}(X_1,k_3h)) - \frac{2 \left (2/d +1 \right ) \log(8n)}{n-1}.
           \]
           Thus we deduce that
           \begin{align*}
           P_n(S(X_1,T,h))
           &\geq \frac{\beta}{2}Q(\mathcal{B}(X_1,k_3h))- \frac{2 \left (2/d +1 \right ) \log(8n)}{n-1}
           \\&\geq \frac{\beta}{2}q(k_3 h)- C_d \frac{\beta h^d}{\kappa}
           \\& \geq h^d  \left ( \beta c_{d,f_{min},\rho}  - C_d \frac{\beta}{\kappa} \right ),
           \end{align*}
           according to Lemma \ref{enoughweight} (since $k_3 \leq 1$). Choosing $\kappa$ large enough (depending on $d$, $\rho$ and $f_{min}$) and then $n$ large enough leads to the result.       
        \end{proof}
\section{Matrix Decomposition and Principal Angles}

In this section we expose a standard matrix perturbation result, adapted to our framework. For real symmetric matrices, we let $\mu_i(\cdot)$ denote their $i$-th largest eigenvalue and $\mu_{min}(\cdot)$ the smallest one. 
\begin{thm}[Sin $\theta$ theorem \cite{Davis70}, this version from Lemma 19 in \cite{Arias13}]\label{angledeviation}
      Let $O \in \R^{D\times D}$, $B\in \R^{d\times d}$ be positive semi-definite symmetric matrices such that 
      \[
      O = \left ( \begin{array}{c|c} 
       B & 0 \\
      \hline
      0 &  0  
\end{array} \right ) + E.
      \]
       Let $T_0$ (resp. $T$) be the vector space spanned by the first $d$ vectors of the canonical basis (resp. by the first $d$ eigenvectors of $O$). Then 
       \[
       \angle \left ( T_0, T \right ) \leq \frac{\sqrt{2}\|E\|_{op}}{\mu_{min}(B)}.
       \]
      \end{thm}

\section{Local PCA for Tangent Space Estimation and Decluttering}\label{localpcaresults}

This section is dedicated to the proofs of Section \ref{tangent_space_estimation_and_denoising_procedure}. We begin with the case of additive noise (and no outliers), that is Proposition \ref{tangent_space_rate_nonoise}.

\subsection{Proof of Proposition \ref{tangent_space_rate_nonoise}}
\label{proofofpropositiontangentspaceratenonoise}
 
Without loss of generality, the local PCA analysis will be conducted at base point $X_1$, the results on the whole sample then follow from a standard union bound. For convenience, we assume that $\pi_M(X_1) = 0$ and that $T_{0}{M}$ is spanned by the $d$ first vectors of the canonical basis of $\mathbb{R}^D$. 
We recall that $X_i = Y_i + Z_i$, with $Y_i \in M$ and $\|Z_i\| \leq s h$, for $s \leq 1/4$. In particular, $\|X_1\| \leq  \|Z_1\| \leq sh \leq h/4$.

We adopt the following notation for the local covariance matrix based on the whole sample $\X_n$.
   \[
   \begin{array}{@{}ccc}
     \hat{\Sigma}(h) &=& \frac{1}{n-1} \sum_{j \geq 2} (X_j - \bar{X}(h))(X_j - \bar{X}(h))^t U (X_i,h), \\
     \bar{X}(h) & = & \frac{1}{N(h)} \sum_{i \geq 2}{X_i U(X_i,h)}, \\
     N(h) &=& \sum_{i \geq 2}{U(X_i,h)}.
     \end{array}
     \]
     Note that the tangent space estimator $\mathtt{TSE}(\X_n,h)_1$ is the space spanned by the first $d$ eigenvectors of $\hat{\Sigma}(h)$. 
 From now on we suppose that all the inequalities of Lemma \ref{Concentration} are satisfied, defining then a global event of probability larger than $1-4 \left( \frac{1}{n} \right )^{2/d + 1}$. 
	 
	 We consider $h=h_0 \leq h_+$, so that Lemma \ref{enoughweight} and \ref{expectations} hold. We may then decompose the local covariance matrix as follows.
	 \begin{align}\label{decomposition}
        \hat{\Sigma}(h) & = \frac{1}{n-1}\sum_{i\geq 2} (X_i - m(h))(X_i-m(h))^t U(X_i,h) - \frac{N(h)}{n-1} ( \bar{X}(h) -m(h))(\bar{X}(h) - m(h))^t \notag \\        
        &:= \hat{\Sigma}_1+\hat{\Sigma}_2. 
        \end{align} 
     The first term may be written as
     \begin{align*}
     \hat{\Sigma}_1 &= \frac{1}{n-1}\sum_{i\geq 2} (X_i - m(h))(X_i-m(h))^t U(X_i,h) \\ &= \frac{1}{n-1}\sum_{i\geq 2} (X_i - m(h))_{\top}(X_i-m(h))_{\top}^t U(X_i,h)+ R_1
      \\ &= \Sigma(h) + R_1 + R_2,
     \end{align*}      
 where
      \[
      \Sigma(h) = \left (\begin{matrix} 
      A(h) & 0 \\
      0 & 0
      \end{matrix} \right ).
      \]
       According to Lemma \ref{expectations} (with $\beta=1$), $\mu_{min}(A(h)) \geq c_d f_{min} h^{d+2}$. On the other hand, using  Proposition \ref{normalpart} and Lemma \ref{expectations} we may write
       \begin{align*}
       (n-1)\|R_1\|_{\mathcal{F}}/N(h) &\leq 2 \sup_{y+z \in \mathcal{B}(X_1,h)}{\|(y+z-m(h))_{\top}\|\|(y+z-m(h))_{\perp}\|} 
       \\
       & \hspace{5ex}+ \sup_{y+z \in \mathcal{B}(X_1,h)}{\|(y+z-m(h))_{\perp}\|^2} \\
       &\leq 2 \sup_{y+z \in \mathcal{B}(X_1,h)}{\|(y+z-m(h))\|\left (\|(y-m(h))_{\perp}\| + s h \right )} 
              \\
       & \hspace{5ex}+ \sup_{y \in \mathcal{B}(0,2h) \cap M}{\left ( \|(y-m(h))_{\perp}\| + s h \right )^2} \\
       & \leq 8h \left ( \frac{4h^2}{\rho} + 2 s h \right ) +  \left ( \frac{4h^2}{\rho} + 2 s h \right )^2 \\ 
       &\leq \frac{34 h^3}{\rho} + 20 s h^2, 
       \end{align*}
         since $h \leq h_+$ and $s \leq 1/4$.
                  In addition, we can write
        \[
      R_2 =   \left (\begin{matrix} 
      R_2 & 0 \\
      0 & 0
      \end{matrix} \right ),
      \]
      with $\|R_2\|_{\mathcal{F}} \leq C_d \frac{f_{max}}{f_{min} \sqrt{\kappa}} q(2h) h^2$ according to Lemma \ref{Concentration} (with $\beta=1$).
      
      In turn, the term $\hat{\Sigma}_2$ may be decomposed as
      \begin{align*}
      \hat{\Sigma}_2 & = \left (\begin{matrix} 
      R_4 & 0 \\
      0 & 0
      \end{matrix} \right ) + R_3,
\end{align*} 
with
\begin{align*}
\|R_4\|_{\mathcal{F}} & \leq \frac{N(h)}{n-1} \|( \bar{X}(h) -m(h))_{\top} \| \|(\bar{X}(h) - m(h))\| \\
                     & \leq \frac{2h}{n-1}\left \| \sum_{i \geq 2} (X_i -m(h))_{\top} U(X_i,h) \right  \| \\
                     & \leq \frac{2C_d q(2h) h^2 f_{max}	}{f_{min} \sqrt{\kappa}},  
\end{align*}
  according to Lemma \ref{Concentration}. A similar bound on $R_3$ may be derived,
   \begin{align*}
   \|R_3\|_{\mathcal{F}} & \leq \frac{N(h)}{n-1} \|( \bar{X}(h) -m(h))_{\perp} \|(\bar{X}(h) - m(h))\| \\
   & \leq \frac{4h}{n-1} \left \| \sum_{i \geq 2} (Y_i + Z_i -m(h))_{\perp} U(X_i,h)  \right \| \\
   & \leq \frac{8hN(h) \left( 2h^2/\rho + s h \right )}{n-1} \\
  & \leq \frac{N(h) h^2}{n-1} \left ( \frac{16h}{\rho} + 8 s \right ),
   \end{align*}
   according to Proposition \ref{normalpart} and Lemma \ref{expectations}. If we choose $h = \left ( \kappa \frac{\log n}{n-1} \right )^{1/d}$, for $\kappa$ large enough (depending on $d$, $f_{min}$ and $f_{max}$), we have
   \begin{align*}
   \frac{\|R_2 + R_4\|_{\mathcal{F}}}{\mu_{min}(A(h))} \leq 1/4.
   \end{align*}
   Now, provided that $\kappa \geq 1$, according to Lemma \ref{Concentration}, we may write
   \[
   \frac{\|R_1 + R_3\|_{\mathcal{F}}}{\mu_{min}(A(h))} \leq K_{f_{max},f_{min},d} \left ( h/ \rho + s \right ),
   \]
   which, for $n$ large enough, leads to 
   \[
   \angle (T_0 M, \hat{T}_{X_1} M) \leq \sqrt{2} K_{f_{max},f_{min},d} \left (  h/ \rho + s \right ),
   \]
   according to Proposition \ref{angledeviation}. 
   
\subsection{Proof of Proposition \ref{noisefiltering}}

          The proof of Proposition \ref{noisefiltering} follows the same path as the derivation of Proposition \ref{tangent_space_rate_nonoise}, with some technical difficulties due to the outliers ($\beta < 1$). We emphasize that in this framework, there is no additive noise ($\sigma =0$).  
 As in the previous section, the analysis will be conducted for $X_1$ $\in$ $\X^{(k)}$, for some fixed $k \geq -1$, $k=-1$ referring to the initialization step. Results on the whole sample then follow from a standard union bound. As before, we assume that $\pi_{{M}}(X_1) = 0$ and that $T_{0}{M}$ is spanned by the $d$ first vectors of the canonical basis of $\mathbb{R}^D$. In what follows, denote by $\hat{t}$ the map from $\mathbb{R}^D$ to $\{0,1\}$ such that $\hat{t}(X_i)=1$ if and only if $X_i$ is in $\X^{(k)}$.
      
   We adopt the following notation for the local covariance matrix based on $\X^{(k)}$ (after $k+1$ iterations of the outlier filtering procedure).
   \[
   \begin{array}{@{}ccc}
     \hat{\Sigma}^{(k)}(h) &=& \frac{1}{n-1} \sum_{j \geq 2} (X_j - \bar{X}(h)^{(k)})(X_j - \bar{X}(h)^{(k)})^t U (X_i,h) \hat{t}(X_i), \\
     \bar{X}^{(k)}(h) & = & \frac{1}{N^{(k)}(h)} \sum_{i \geq 2}{X_i U(X_i,h) \hat{t}(X_i)}, \\
     N^{(k)}(h) &=& \sum_{i \geq 2}{U(X_i,h) \hat{t}(X_i)}.
     \end{array}
     \]
     Also recall that we define $N_0(h)$ and $N_1(h)$ as the number of points drawn from  respectively clutter and signal in $\mathcal{B}(X_1,h) \cap M$ (based on the whole sample $\X_n$).
	At last, we suppose that all the inequalities of Lemma \ref{Concentration} and Lemma \ref{universalthreshold} are satisfied, defining then a global event of probability larger than $1-7 \left( \frac{1}{n} \right )^{2/d + 1}$.
	 
	 We recall that we consider  $h_\infty \leq h \leq h_{k}$,  $k \geq -1$ (with $h_{-1} = h_0$), and $X_1$ in $\X^{(k)}$ such that $\dd(X_1,M) \leq h/\sqrt{2}$. We may then decompose the local covariance matrix as
	 \begin{align}\label{decomposition2}
        \hat{\Sigma}^{(k)}(h)& = \frac{1}{n-1}\sum_{i\geq 2} (X_i - m(h))(X_i-m(h))^t U(X_i,h) \hat{t}(X_i) 
               \\
       & \hspace{15ex}
       - \frac{N^{(k)}(h)}{n-1} ( \bar{X}^{(k)}(h) -m(h))(\bar{X}(h)^{(k)} - m(h))^t \notag \\
        & = \begin{multlined}[t] \frac{1}{n-1}\sum_{i\geq 2} (X_i - m(h))(X_i-m(h))^t U(X_i,h) \hat{t}(X_i)V_i (X_i) 
               \\
       + \frac{1}{n-1}\sum_{i\geq 2} (X_i - m(h))(X_i-m(h))^t U(X_i,h) (1-V_i) \hat{t}(X_i) \\ - \frac{N^{(k)}(h)}{n-1} ( \bar{X}(h)^{(k)} -m(h))(\bar{X}(h)^{(k)} - m(h))^t,
        \end{multlined}\notag \\
        &:= \hat{\Sigma}^{(k)}_1+\hat{\Sigma}^{(k)}_2+\hat{\Sigma}^{(k)}_3. 
        \end{align}         
        The proof of Proposition \ref{noisefiltering} will follow by induction.
  
   \noindent\textbf{Initialization step ($k=-1$)}:         

    In this case $\X^{(k)} = \X_n$, $h=h_0$, $\dd(X_1,M) \leq h_0/\sqrt{2}$, and $\hat{t}$ is always equal to $1$. Then the first term $\hat{\Sigma}_1^{(k)}$ of \eqref{decomposition2} may be written as
     \begin{align*}
     \frac{1}{n-1}\sum_{i\geq 2} (X_i - m(h))(X_i-m(h))^t U(X_i,h) V_i &= \frac{1}{n-1}\sum_{i\geq 2} (X_i - m(h))_{\top}(X_i-m(h))_{\top}^t U(X_i,h) V_i + R_1 \\
      &= \Sigma(h) + R_1 + R_2,
     \end{align*}      
 where
      \[
      \Sigma(h) = \left (\begin{matrix} 
      A(h) & 0 \\
      0 & 0
      \end{matrix} \right ).
      \]
       According to Lemma \ref{expectations}, $\mu_{min}(A(h)) \geq c_d f_{min} \beta h^{d+2}$, and $\|R_1\|_{\mathcal{F}} \leq 34 \frac{N_1(h) h^3	}{\rho(n-1)}$ according to Proposition \ref{normalpart}. Moreover, we can write
      \[
      R_2 =   \left (\begin{matrix} 
      R_2 & 0 \\
      0 & 0
      \end{matrix} \right ),
      \]
      with $\|R_2\|_{\mathcal{F}} \leq C_d \frac{f_{max}}{f_{min} \sqrt{\kappa}} \beta q(2h) h^2$ according to Lemma \ref{Concentration}.
       
    Term $\hat{\Sigma}^{(k)}_2$ in inequality \eqref{decomposition2} may be bounded by
   \begin{align*}
      \|\hat{\Sigma}^{(k)}_2\|_{\mathcal{F}} \leq \frac{16 h^2 N_0(h)}{n-1}.
\end{align*}       
In turn, term $\hat{\Sigma}^{(k)}_3$ may be decomposed as
\begin{align*}
\frac{N^{(k)}(h)}{n-1} ( \bar{X}(h)^{(k)} -m(h))(\bar{X}(h)^{(k)} - m(h))^t &=  \left (\begin{matrix} 
      R_6 & 0 \\
      0 & 0
      \end{matrix} \right ) + R_5,
\end{align*} 
with
\begin{align*}
\|R_6\|_{\mathcal{F}} & \leq \frac{N(h)^{(k)}}{n-1} \|( \bar{X}(h)^{(k)} -m(h))_{\top} \| \|(\bar{X}(h)^{(k)} - m(h))\| \\
                     & \leq \frac{4h}{n-1}\left (\| \sum_{i \geq 2} (X_i -m(h))_{\top} U(X_i,h) V_i \| + \| \sum_{i \geq 2} (X_i -m(h))_{\top} U(X_i,h) (1-V_i) \| \right ) \\
                     & \leq \frac{4C_d \beta q(2h) h^2 f_{max}	}{f_{min} \sqrt{\kappa}} + \frac{16 h^2 N_0(h)}{n-1},  
\end{align*}
   according to Lemma \ref{Concentration}. We may also write
   \begin{align*}
   \|R_5\|_{\mathcal{F}} & \leq \frac{N(h)^{(k)}}{n-1} \|( \bar{X}(h)^{(k)} -m(h))_{\perp} \| \|(\bar{X}(h)^{(k)} - m(h))\| \\
   & \leq \frac{4h}{n-1} \left (\| \sum_{i \geq 2} (X_i -m(h))_{\perp} U(X_i,h) V_i \| + \| \sum_{i \geq 2} (X_i -m(h))_{\perp} U(X_i,h)(1-V_i) \| \right ) \\
  & \leq \frac{16N_1(h) h^3}{(n-1) \rho} + \frac{16 N_0(h) h^2 }{(n-1)},
   \end{align*}
   according to Proposition \ref{normalpart} and Lemma \ref{expectations}. As in the additive noise case (see proof of Proposition \ref{tangent_space_rate_nonoise}), provided that $\kappa$ is large enough (depending on $d$, $f_{min}$, and $f_{max}$), we have
   \[
   \frac{\|R_2 + R_6\|_{\mathcal{F}}}{\mu_{min}(A(h))} \leq 1/4.
   \]
   Since $(n-1)h_0^{d} =\frac{\kappa \log n}{\beta h}$, if we ask $\kappa \geq \rho$, then for $n$ large enough we eventually get
   \[
   \frac{\|\hat{\Sigma}^{(k)}_2 + R_1 + R_5\|_{\mathcal{F}}}{\mu_{min}(A(h))} \leq K_{d,f_{min},f_{max},\beta} \frac{h_0}{\rho},
   \]
   according to Lemma \ref{Concentration}. Then, Proposition \ref{angledeviation} can be applied to obtain 
   \[
    \angle ( \mathtt{TSE}(\X^{(-1)},h_{0})_1, T_{\pi(X_1)} {M} )\leq \sqrt{2} K^{(0)}_{d,f_{min},f_{max},\beta} h_0/\rho.
   \]
   According to Lemma \ref{universalthreshold}, we may choose $\kappa$ large enough (with respect to $K=\sqrt{2} K^{(0)}$, $d$, $f_{min}$ and $\rho$) and then a threshold $t$ so that, if $X_1 \in M$, then $X_1 \in \X^{(0)}$, and if $d(X_1,M) \geq h_0^2/ \rho$, then $X_1 \notin \X^{(0)}$.      
   
   \noindent\textbf{Iteration step}  
   Now we assume that $k \geq 0$, and that $\dd(X_i,M) \geq h_k^2/\rho$ implies $\hat{t}(X_i)=0$, with $h_k = \left ( \kappa \frac{\log n}{\beta(n-1)} \right )^{\gamma_k}$, $\gamma_k$ being between $1/(d+1)$ and $1/d$. Let $h_{\infty} \leq h \leq h_k$, and suppose that $\dd(X_1,M) \leq h_k/\sqrt{2}$. As in the initialization step, $\hat{\Sigma}_1^{(k)}$ may be written as
\begin{align*}
\left (\begin{matrix} 
      A(h) & 0 \\
      0 & 0
      \end{matrix} \right ) + R_1 + R_2,
\end{align*}   
with $\mu_{min}(A(h)) \geq c_d f_{min} \beta h^{d+2}$, $\|R_1\|_{\mathcal{F}} \leq 34 \frac{N_1(h) h^3	}{\rho(n-1)}$, and $\|R_2\|_{\mathcal{F}} \leq C_d \frac{f_{max}}{f_{min} \sqrt{\kappa}} \beta q(2h) h^2$.
  
   We can decompose $\hat{\Sigma}_2$ as 
   \begin{align*}
      &\frac{1}{n-1}\sum_{i\geq 2} (X_i - m(h))(X_i-m(h))^t U(X_i,h)(1-V_i) \hat{t}(X_i) \\
      \hspace{3ex}& = \frac{1}{n-1}\sum_{i\geq 2} (X_i - m(h))_{\top}(X_i-m(h))_{\top}^t U(X_i,h) (1-V_i) \hat{t}(X_i) + R_3 \\
       & = \left (\begin{matrix} 
      R_4 & 0 \\
      0 & 0
      \end{matrix} \right ) + R_3,
\end{align*}       
with $\|R_4\|_{\mathcal{F}} \leq \frac{16N_0(h) h^2}{n-1}$ and $\|R_3\| \leq \frac{128 N_0(h) h h_k^2}{(n-1)\rho}$, according to Proposition \ref{normalnoise}, for $n$ large enough so that $h_0^2/\rho \leq h_{\infty}$. Term $\hat{\Sigma}^{(k)}_3$ may also be written as
\begin{align*}
\frac{N(h)^{(k)}}{n-1} ( \bar{X}(h)^{(k)} -m(h))(\bar{X}(h)^{(k)} - m(h))^t &=  \left (\begin{matrix} 
      R_6 & 0 \\
      0 & 0
      \end{matrix} \right ) + R_5,
\end{align*} 
with
\begin{align*}
\|R_6\|_{\mathcal{F}} & \leq \frac{N(h)^{(k)}}{n-1} \|( \bar{X}(h)^{(k)} -m(h))_{\top}\| \|(\bar{X}(h)^{(k)} - m(h))\| \\
                     & \leq \frac{4h}{n-1}\left (\| \sum_{i \geq 2} (X_i -m(h))_{\top} U(X_i,h) V_i \| + \| \sum_{i \geq 2} (X_i -m(h))_{\top} U(X_i,h) (1-V_i)\hat{t}(X_i) \| \right ) \\
                     & \leq \frac{4C_d \beta q(2h) h^2 f_{max}	}{f_{min} \sqrt{\kappa}} + \frac{16h^2 N_0(h)}{(n-1)},  
\end{align*}
   according to Lemma \ref{Concentration}. We may also write
   \begin{align*}
   \|R_5\|_{\mathcal{F}} & \leq \frac{N(h)^{(k)}}{n-1} \|( \bar{X}(h)^{(k)} -m(h))_{\perp} \|\|(\bar{X}(h)^{(k)} - m(h))\| \\
   & \leq \frac{4h}{n-1} \left (\| \sum_{i \geq 2} (X_i -m(h))_{\perp} U(X_i,h) V_i \| + \| \sum_{i \geq 2} (X_i -m(h))_{\perp} U(X_i,h)(1-V_i)\hat{t}(X_i) \| \right ) \\
  & \leq \frac{16 N_1(h) h^3}{(n-1) \rho} + \frac{32 N_0(h) h h_k^2 }{\rho(n-1)},
   \end{align*}
   according to Proposition \ref{normalpart}, Proposition \ref{normalnoise} and Lemma \ref{expectations}. As done before, we may choose $\kappa$ large enough (depending on $d$, $f_{min}$ and $f_{max}$, but not on $k$) such that
   \[
   \frac{\|R_2 + R_4 + R_6\|_{\mathcal{F}}}{\mu_{min}(A(h))} \leq 1/4.
   \]  
    Now choose $h = h_{k+1} = \left  ( \kappa \frac{\log n}{\beta(n-1)} \right )^{(2\gamma_k +1)/(d+2)}$, with $\kappa \geq 1$. This choice is made to optimize residual terms of the form $h/\rho + h_k^2 N_0(h)/h$ coming from ${\| R_1 + R_3 + R_5 \|_{\mathcal{F}}}/{\mu_{min}(A(h_{k+1}))}$. Then we get, according to Lemma \ref{Concentration},
    \begin{align}
    \frac{\| R_1 + R_3 + R_5 \|_{\mathcal{F}}}{\mu_{min}(A(h_{k+1}))} & \leq C_d \frac{f_{max}h_{k+1}}{\rho f_{min}} + \frac{C'_d}{\beta \rho f_{min}} \left ( \kappa \frac{\log n}{\beta(n-1)} \right )^{ \gamma_{k+1} + 2 \gamma_k - (2\gamma_k +1) + 1} \label{biasvariance}\\
     & \leq K_{d,f_{min},f_{max},\beta} \frac{h_{k+1}}{\rho}, \notag
    \end{align}
    where again, $K_{d,f_{min},f_{max},\beta}$ does not depend on $k$.
    At last, we may apply Proposition \ref{angledeviation} to get
    \begin{align*}
    \angle ( \mathtt{TSE}(\X^{(k)},h_{k+1})_1, T_{\pi(X_1)} {M} ) 
    &\leq 
    \sqrt{2} K_{d,f_{min},f_{max},\beta} h_{k+1}/\rho 
    \\
    &\leq 
    \sqrt{2} \left( K_{d,f_{min},f_{max},\beta} \vee K^{(0)}_{d,f_{min},f_{max},\beta}\right) h_{k+1}/\rho
    \\
    &:= C_{d,\beta,f_{max},f_{min}} h_{k+1}/\rho.
    \end{align*}
  
Then, according to Lemma \ref{universalthreshold}, we may choose $\kappa$ large enough (not depending on $k$) and $t$ (not depending on $k$ either) so that if $X_1 \in M$, then $X_1 \in \X^{(k+1)}$, and if $d(X_1,M) \geq h_{k}^2/\rho$, then $X_1 \notin \X^{(k+1)}$.      
   Proposition \ref{noisefiltering} then follows from a straightforward union bound on the sample $\{X_1, \hdots, X_n\}$.

\subsection{Proof of Proposition \ref{infinite_denoising_estimator}}

In this case, we have $\dd(X_j,M) \leq h_\infty^2/\rho$, for every $X_j$ in $\X^{(\hat{k})}$. The proof of Proposition \ref{infinite_denoising_estimator} follows from the same calculation as in the proof of Proposition \ref{noisefiltering}, replacing $h_k^2/\rho$ by its upper bound $h_{\infty}^2/\rho$ and taking $h_{k+1} = h_{\infty}$ in the iteration step.
        
\section{Proof of the Main Reconstruction Results}
{\color{black}
We now prove main results Theorem \ref{main_result_noob_nonoise} (additive noise model), and Theorems \ref{main_result_noob_noise} and \ref{theoretical_rate} (clutter noise model).
}
\subsection{Additive Noise Model}

\begin{proof}[\proofof Corollary \ref{recap_nonoise}]
Let $ Q \in \mathcal{G}_{D,d,f_{min},f_{max},\rho,\sigma}$. 
Write $\varepsilon = c_{d,f_{min},f_{max}} (h \vee \rho \sigma/h)$ for $c_{d,f_{min},f_{max}}$ large enough, an consider the event $A$ defined by
\[
A = 
\begin{multlined}[t]
\left\{ 
	\max_{X_j \in \X_n} \angle ( T_{\pi_M(X_j)} {M} , \hat{T}_j(h)) 
	\leq 
	C_{d,f_{min},f_{max}} \left( \frac{h}{\rho} + \frac{\sigma}{h} \right)
\right\} 
\cap 
\left\{ 
	\sup_{x \in M} d(x,\X_n) \leq \sigma
\right\}
\\
\cap
\left\{ 
	\sup_{X_j \in \X_n} d(X_j,M) 
	\leq 
	C_{d,f_{min}} \left( \frac{\log n}{n} \right)^{1/d}
\right\}
.
\end{multlined}
\]
Then from Proposition \ref{tangent_space_rate_nonoise} and Lemma \ref{epsilon_rate}, $\p_Q(A) \geq 1 - 5 \left( \frac{1}{n} \right)^{2/d}$, and from the definition of $\varepsilon$ and the construction of $\Y_n$, for $n$ large enough,
\begin{align*}
A 
&\subset 
\begin{multlined}[t]
\left\{ 
	\max_{X_j \in \X_n} \angle ( T_{\pi_M(X_j)} {M} , \hat{T}_j(h)) 
	\leq
	\frac{\varepsilon}{2280 \rho}
\right\}
\cap
\left\{
	\sup_{x \in M} d(x,\X_n) \leq \varepsilon
\right\}
\\
\cap
\left\{ 
	\sup_{X_j \in \X_n} d(X_j,M) \leq \frac{\varepsilon^2}{1140 \rho}
\right\}
\end{multlined}
\\
&\subset
\begin{multlined}[t]
\left\{ 
	\max_{X_j \in \Y_n} \angle ( T_{\pi_M(X_j)} {M} , \hat{T}_j(h)) 
	\leq
	\frac{\varepsilon}{2280 \rho}
\right\}
\cap
\left\{
	\sup_{x \in M} d(x,\Y_n) \leq 2\varepsilon
\right\}
\\
\cap
\left\{
	\Y_n \text{~is~} \varepsilon \text{-sparse}
\right\}
\cap
\left\{ 
	\sup_{X_j \in \Y_n} d(X_j,M) \leq \frac{\varepsilon^2}{1140 \rho}
\right\}
,
\end{multlined}
\end{align*}
which yields the result.
\end{proof}

\begin{proof}[\proofof Theorem \ref{main_result_noob_nonoise}]
Following the above notation, we observe that on the event $A$, Theorem \ref{tdc_stability} holds for $\varepsilon = c_{d,f_{min},f_{max}} (h \vee \rho \sigma/h)$, $\theta = \varepsilon/(1140 \rho)$ (where we used that $\theta \leq 2 \sin \theta$) and $\eta = \varepsilon^2/(1140 \rho)$ with high probability, so that the first part of Theorem \ref{main_result_noob_nonoise} is proved.
Furthermore, for $n$ large enough, 
\begin{align*}
\E_Q\left[ \dd_{H}\left(M,\hat{M}_{\mathtt{TDC}}\right)\right]
	&=  \E_Q\left[ \dd_{H}\left(M,\hat{M}_{\mathtt{TDC}}\right)\mathbbm{1}_A\right]
			+
		\E_Q\left[ \dd_{H}\left(M,\hat{M}_{\mathtt{TDC}}\right)\mathbbm{1}_{A^c}\right] 
	\\
	&\leq 
	C_d \frac{\varepsilon^2}{\rho}
	+
	\left(1 - \p_Q(A) \right)\left( \mathrm{diam}(M) + \sigma \right) 
	\\
	&\leq C'_{d,f_{min},f_{max},\rho} \varepsilon^2,
\end{align*}
where for the last line we used the diameter bound of Lemma \ref{model_properties}.
\end{proof}

\subsection{Clutter Noise Model}

\begin{proof}[\proofof Corollary \ref{recap_noise}]
Let $ P \in \mathcal{O}_{D,d,f_{min},f_{max},\rho,\beta}$. 
For $n$ large enough, write $\varepsilon = c_{d,f_{min}f_{max}} h_{k_\delta}$ for $c_{d,f_{min}f_{max}}$ large enough,
and consider the event
\[
A^\delta 
= 
\begin{multlined}[t]
\left\{ 
	\max_{X_j \in \X^{(k_\delta)}} \angle ( T_{\pi_M(X_j)} {M} , \hat{T}^\delta_j) 
	\leq 
	C_{d,f_{min},f_{max}} \frac{h_{k_\delta}}{\rho}	
	\right\} 
\cap 
\left\{ 
	\sup_{x \in M} d(x,\X^{(k_\delta)}) \leq \frac{h_{k_\delta}^2}{\rho}
\right\}
\\
\cap
\left\{ 
	\sup_{X_j \in \X^{(k_\delta)}} d(X_j,M)
	\leq 
	C_{d,f_{min}} \left( \frac{\log n}{n} \right)^{1/d}
\right\}
.
\end{multlined}
\]
From Proposition \ref{noisefiltering} and Lemma \ref{epsilon_rate}, $\p_P\bigl(A^\delta\bigr) \geq 1- 8 \left( \frac{1}{n} \right)^{2/d}$and from the definition of $\varepsilon$ and the construction of $\Y_n^\delta$, for $n$ large enough,

\begin{align*}
A^\delta
&\subset 
\begin{multlined}[t]
\left\{ 
	\max_{X_j \in \X^{(k_\delta)}} \angle ( T_{\pi_M(X_j)} {M} , \hat{T}_j^\delta) 
	\leq
	\frac{\varepsilon}{2280 \rho}
\right\}
\cap
\left\{
	\sup_{x \in M} d(x,\X^{(k_\delta)}) \leq \varepsilon
\right\}
\\
\cap
\left\{ 
	\sup_{X_j \in \X^{(k_\delta)}} d(X_j,M) \leq \frac{\varepsilon^2}{1140 \rho}
\right\}
\end{multlined}
\\
&\subset
\begin{multlined}[t]
\left\{ 
	\max_{X_j \in \Y_n^\delta} \angle ( T_{\pi_M(X_j)} {M} , \hat{T}_j^\delta) 
	\leq
	\frac{\varepsilon}{2280 \rho}
\right\}
\cap
\left\{
	\sup_{x \in M} d(x,\Y_n^\delta) \leq 2\varepsilon
\right\}
\\
\cap
\left\{
	\Y_n \text{~is~} \varepsilon \text{-sparse}
\right\}
\cap
\left\{ 
	\sup_{X_j \in \Y_n^\delta} d(X_j,M) \leq \frac{\varepsilon^2}{1140 \rho}
\right\}
,
\end{multlined}
\end{align*}
which yields the result.

\end{proof}

\begin{proof}[\proofof Theorem \ref{main_result_noob_noise}]
Following the above notation, we observe that on the event $A^\delta$, Theorem \ref{tdc_stability} holds for $\varepsilon = c_{d,f_{min}f_{max}} h_{k_\delta}$, $\theta = \varepsilon/(1140 \rho)$ and $\eta = \varepsilon^2/(1140 \rho)$,
so that the first part of Theorem \ref{main_result_noob_noise} is proved.
As a consequence, for $n$ large enough, 
\begin{align*}
\E_P\left[ \dd_{H}\left(M,\hat{M}_{\mathtt{TDC \delta}}\right)\right]
	&=
	\E_P\left[ \dd_{H}\left(M,\hat{M}_{\mathtt{TDC \delta}}\right)\mathbbm{1}_{A^\delta}\right]
	+
	\E_P\left[ \dd_{H}\left(M,\hat{M}_{\mathtt{TDC \delta }}\right)\mathbbm{1}_{\left(A^{\delta}\right)^c}\right]
	\\
	&\leq 
	C_d\frac{\varepsilon^2}{\rho}
	+
	\left(1- \p_P\bigl(A^{\delta}\bigr)\right) \times 2K_0
	\\
	&\leq C'_{d,f_{min},f_{max},\rho} \varepsilon^2,
\end{align*}
where for the second line we used the fact that $M \cup \hat{M}_\mathtt{TDC \delta } \subset \mathcal{B}_0$, a ball of radius $K_0 = K_0(d,f_{min},\rho)$.
\end{proof}

Finally, Theorem \ref{theoretical_rate} is obtained similarly using Proposition \ref{infinite_denoising_estimator}.


\end{document}